\documentclass[a4paper,11pt,reqno]{amsart}
\usepackage{amsmath}
\usepackage{amsthm,enumerate}

\usepackage{graphicx}
\usepackage{amssymb}

\usepackage{appendix}

\usepackage[italian,english]{babel}

\usepackage[utf8x]{inputenc}
\usepackage{fancyhdr}

\usepackage{calc}
\usepackage{url}

\usepackage[text={6.3in,8.6in},centering]{geometry}

\usepackage{srcltx, inputenc}

\usepackage[T1]{fontenc}

\usepackage[usenames,dvipsnames]{color}

\makeatletter
\def\cleardoublepage{\clearpage\if@twoside \ifodd\c@page\else%
         \hbox{}%
     \thispagestyle{empty}%              % Empty header styles
     \newpage%
     \if@twocolumn\hbox{}\newpage\fi\fi\fi}
\makeatother

\let\cleardoublepage\clearpage

\addto\captionsitalian{}

\newtheorem{thm}{Theorem}[section]
\newtheorem{cor}[thm]{Corollary}
\newtheorem{lem}[thm]{Lemma}
\newtheorem{pro}[thm]{Proposition}
\newtheorem{den}[thm]{Definition}
\newtheorem{oss}[thm]{Remark}
\numberwithin{equation}{section}

\newcommand{\V}{\mathcal{V}}
\newcommand{\s}{\mathcal{S}}

\begin{document}

\title[A general nonlinear characterization of stochastic incompleteness]{A general nonlinear characterization\\ of stochastic incompleteness}

\author[Gabriele Grillo]{Gabriele Grillo*}
\address{\hbox{\parbox{5.7in}{\medskip\noindent{Gabriele Grillo, Matteo Muratori and Fabio Punzo: \\  Dipartimento di Matematica,\\
Politecnico di Milano,\\
   Piazza Leonardo da Vinci 32, 20133 Milano, Italy. \smallskip }}}}
        \email{gabriele.grillo@polimi.it}
      \email{matteo.muratori@polimi.it}
       \email{fabio.punzo@polimi.it}

\author{Kazuhiro Ishige}
\address{\hbox{\parbox{5.7in}{\medskip\noindent{Kazuhiro Ishige: \\ Graduate School of Mathematical Sciences,\\
The University of Tokyo,\\
  3-8-1 Komaba, Meguro-ku, Tokyo 153-8914, Japan. \smallskip }}}}
\email{ishige@ms.u-tokyo.ac.jp}

\author[Matteo Muratori]{Matteo Muratori}

\author{Fabio Punzo}

	\makeatletter
	\@namedef{subjclassname@2020}{\textup{2020} Mathematics Subject Classification}
	\makeatother


\keywords{Stochastic incompleteness; noncompact manifolds; nonlinear elliptic equations; nonlinear diffusion equations; uniqueness and nonuniqueness properties. \\ \indent *Corresponding author.}

\maketitle

\begin{abstract}
Stochastic incompleteness of a Riemannian manifold $M$ amounts to the nonconservation of probability for the heat semigroup on $M$. We show that this property is equivalent to the existence of nonnegative, nontrivial, bounded (sub)solutions to $\Delta W=\psi(W)$ for one, hence all, general nonlinearity $\psi$ which is only required to be continuous, nondecreasing, with $\psi(0)=0$ and $\psi>0$ in $(0,+\infty)$. Similar statements hold for (sub)solutions that may change sign. We also prove that stochastic incompleteness is equivalent to the nonuniqueness of bounded solutions to the nonlinear parabolic equation $\partial_t u =\Delta\phi(u)$ with bounded initial data for one, hence all, general nonlinearity $\phi$ which is only required to be continuous, nondecreasing and nonconstant. Such a generality allows us to deal with equations of both fast-diffusion and porous-medium type, as well as with the one-phase and two-phase classical Stefan problems, which seem to have never been investigated in the manifold setting.

%%%%%%%%%%%%%%

\vskip8pt
\noindent\textsc{Résumé}.  L'incomplétude stochastique d'une variété riemannienne $M$
équivaut à la non-con\-ser\-vation de la probabilité pour
le semi-groupe de la chaleur sur $M$. Nous montrons que cette propriété est équivalente
à l'existence de (sous-)\-solutions bornées, non négatives
et non triviales de l'équation $  \Delta W = \psi(W)$ 
pour une, et donc pour toute, non-linéarité générale $\psi$,
supposée simplement continue, croissante,
avec $\psi(0) = 0$ et $\psi > 0$ sur $(0, +\infty)$.
Des énoncés analogues valent pour des (sous-)solutions
qui peuvent changer de signe. Nous démontrons également que l'incomplétude stochastique
est équivalente à la non-unicité de solutions bornées
de l'équation parabolique non linéaire
$\partial_t u = \Delta \varphi(u)$ 
avec des données initiales bornées, pour une, et donc
pour toute, non-linéarité générale $\varphi$ supposée
simplement continue, croissante et non constante. Un tel degré de généralité nous permet d'aborder
des équations à la fois de type diffusion rapide
et de type milieu poreux, ainsi que les problèmes classiques
de Stefan à une ou deux phases, qui, à notre connaissance,
n'ont encore jamais été étudiés dans le contexte des variétés.

\end{abstract}

\section{Introduction}

Let $M$ be a noncompact Riemannian manifold of dimension $N\ge2$, not necessarily complete. It is said to be \emph{stochastically complete} if $T_t1=1$ for every $t>0$, where $T_t$ is the \emph{heat semigroup} on $M$. In terms of the minimal heat kernel $k(t,x,y)$ of $M$, this property reads
\begin{equation}\label{cons}
\int_M k(t,x,y) \, d\V(y) =1  \qquad \forall (x,t) \in M \times (0,+\infty) \, ,
\end{equation}
where $ \V $ is the Riemannian volume measure of $M$, and can equivalently be stated by asserting that the lifetime of \emph{Brownian paths} is a.s.~infinite. On the contrary, $M$ is said to be \emph{stochastically incomplete} if the above integral is strictly smaller than $ 1 $ for some, hence all, $ (x,t) \in M \times (0,+\infty) $. We stress that our treatment of stochastic completeness will be purely analytic, through properties of the heat equation and the associated heat kernel, rather than explicit probabilistic methods. We assume $M$ to have empty boundary, nonetheless $M$ need not be complete. 

In fact, there exist several analytic, geometric and probabilistic criteria ensuring that a manifold is stochastically (in)complete. Since it is hopeless to give an exhaustive account of the literature, we limit ourselves to quoting the following relevant papers, without any claim to completeness: \cite{A, D, G, GrigTo, H, I1, I, IM, K, MV, M, M2, PRS2, PRS,  S, T, Y}. They involve, for example, curvature conditions, in the form of sufficiently fast divergence of the sectional curvatures to $-\infty$ at spatial infinity (stochastic incompleteness), or volume growth conditions, in the form of not-too-fast divergence of the volume of geodesic balls as a function of their radius (stochastic completeness). We refer to A.~Grigor'yan's works \cite{GrigBams, Grig, GrigHK} for more comprehensive reviews; see also \cite[Appendix B]{GIM} for a concise reminder of some of such criteria.

In order to properly introduce our work, it is of particular importance to comment that, again according to a result of Grigor'yan, stochastic completeness of a manifold $M$ turns out to be equivalent to any of the following two conditions:
\begin{itemize}

\item For all $ T \in (0,+\infty] $ and all $u_0\in L^\infty(M)$, the Cauchy problem
\begin{equation}\label{calore}
\begin{cases}
{\partial_t u}=\Delta u &\text{in}\ M\times(0,T) \, , \\
u = u_0 & \text{on} \ M \times \{0\} \, ,
\end{cases}
\end{equation}
admits a unique solution in $L^\infty(M\times(0,T))$;

\smallskip

\item For all $\lambda>0$ the equation $\Delta U = \lambda U$ does not admit any (nonnegative) nontrivial bounded solution.

\end{itemize}

It can also be shown, using for instance \cite[Theorem 8.18]{GrigHK}, that stochastic \it incompleteness \rm is in turn equivalent to any of the following conditions, which in principle is stronger than the mere negation of the above properties:

\begin{itemize}

\item For all $ T \in (0,+\infty] $ and all $u_0\in L^\infty(M)$, the Cauchy problem \eqref{calore} admits at least two solutions in $L^\infty(M\times(0,T))$;

\smallskip

\item For all $\lambda>0$ the equation $\Delta U = \lambda \, U$ admits a (nonnegative) nontrivial bounded solution.

\end{itemize}
Note that \eqref{cons} amounts to the fact that, in the special case $ u_0=1 $, the unique bounded solution to \eqref{calore} is the constant $1$.

\smallskip

The goal of the present paper is to prove general \emph{nonlinear} analogues of the above characterization results. In particular, we will show that stochastic \it incompleteness \rm is equivalent to the \it existence of a nonnegative, nontrivial, bounded (sub)solution \rm to the nonlinear elliptic equation
\begin{equation}\label{eqe-intro}
\Delta W=\psi(W) \qquad \text{in } M \, ,
\end{equation}
for one, hence all, nonlinearity $\psi$ only assumed to be continuous in $[0,+\infty)$, nondecreasing, such that $\psi(0)=0$ and $ \psi>0 $ in $ (0,+\infty) $, and in fact to the existence of \it infinitely many \rm such solutions. This is the content of Theorem \ref{teo-ell}. The assumptions on the nonlinearity $\psi$ are essentially optimal, see the discussion in Remark \ref{rr}. As a direct consequence, we have that the existence of a bounded (sub)solution to \eqref{eqe-intro} as above for \emph{one} $\psi$ {ensures} that same is true for \emph{all} $\psi$. Similar statements hold for (sub)solutions that may change sign as well, see Theorem \ref{cor-gen-prs}. It is important to stress that \it boundedness \rm  is unavoidable in order to get an \it equivalence \rm result: even in $\mathbb R^N$, which is stochastically complete, it is well known that there are specific nonlinearities for which \eqref{eqe-intro} admits positive \it unbounded \rm solutions, see \cite{GMP, GMP-dcds} for some examples in the Riemannian setting as well, and \cite{MRS} for a self-contained reminder of the Euclidean case and the associated Keller-Osserman conditions, along with suitable Riemannian generalizations.

\smallskip

Besides, in the evolutionary setting, we will consider the \it filtration equation \rm
\begin{equation}\label{pp1}
\begin{cases}
{\partial_t u}=\Delta \phi(u) & \text{in } M\times(0,T) \, , \\
u = u_0 & \text{on } M \times \{ 0 \}  \, ,
\end{cases}
\end{equation}
with $ T \in (0,+\infty] $ and $u_0\in L^\infty(M)$ (possibly sign changing), where  $\phi$ is only assumed to be a continuous, nondecreasing and nonconstant function. We will show that stochastic \it incompleteness \rm is equivalent to the \it nonuniqueness of solutions in $ L^\infty(M \times (0,T)) $ to \eqref{pp1} for one, hence all, set of data $ \phi$, $ T$, $u_0 $. \rm This is the content of Theorem \ref{teo-par}. Note that \it strict \rm monotonicity is \it not \rm required. It is remarkable that this result covers simultaneously degenerate and singular nonlinear evolutions such as the \emph{porous medium equation} and \emph{fast diffusion equation} (that is $\phi(u)=|u|^{m-1}u$ with $m>1$ and, respectively, $m\in(0,1)$), the one-phase classical \emph{Stefan problem}, which corresponds to $\phi(u)=(u-L)^+$ for some $L>0$ (the latent heat), $u$ being the enthalpy and $\phi(u)$ the temperature, along with its two-phase generalization, corresponding to $\phi(u)=(u-L_1)^+-(u+L_2)^-$ for some parameters $L_1, L_2>0$.

\smallskip

In order to better explain our contribution, let us briefly discuss what was previously available in the literature. As concerns the elliptic setting, Grigor'yan's results have been widely generalized by Pigola, Rigoli and Setti in \cite{PRS2, PRS}. In particular, it is shown there that stochastic \emph{completeness} is equivalent to the nonexistence of suitable subsolutions to
\begin{equation}\label{ell}
\Delta w=f(w) \qquad \text{in } M \, ,
\end{equation}
for any $ f \in C^0(\mathbb{R})$ (not necessarily monotone), where $ w $ is required to be of class $C^2$, bounded from above and to satisfy an unavoidable condition involving the supremum of $ w^+ $. This is a very strong statement, since the assumptions on $f$ are quite general and the result involves subsolutions only. It should however be observed that:

\begin{enumerate}
\item It is not possible, from the above result, to gain information on the existence of subsolutions to \eqref{ell} when $M$ is stochastically \emph{incomplete}, for an arbitrarily given $f$. For the same reasons,  one cannot conclude that if \eqref{eqe-intro} admits no nonnegative, nontrivial, bounded solution for a \emph{single} nonlinearity $\psi$, then the same is true for all nonlinearities, and hence $M$ is stochastically complete.
\smallskip
\item The required regularity assumption on $w$ is restrictive, especially for subsolutions.
\end{enumerate}
We will deal with precise answers the points raised above in one of our main results, Theorem \ref{teo-ell}. A further extension of some of the relevant statements of \cite{PRS2, PRS} is also provided in Theorem \ref{cor-gen-prs}, see  the subsequent Remark \ref{setti}.

\smallskip
 In the recent paper \cite{GIM}, these results have been improved, although for a significantly more restrictive class of nonlinearities. In fact, it is shown there that stochastic completeness is equivalent to the nonexistence of nonnegative, nontrivial, bounded solutions to nonlinear elliptic equations of the form \eqref{eqe-intro} with $\psi$ \emph{convex}, strictly increasing (with $ \psi(0)=0 $) and $ C^1 $. This settles item (1) above, under the stated additional assumptions on $f=\psi$. We also comment that item (2) is dealt with as well, since solutions are just required to be very weak (i.e.~solutions in the sense of distributions) and the methods employed do not in principle require continuity, although elliptic regularity of course yields more than that. It should however be mentioned that results for subsolutions are not explicitly provided in \cite{GIM}, even if they may be deduced from the corresponding techniques of proof. In this framework, the prototype nonlinearity for \eqref{eqe-intro} is $\psi(W)=W^p$ with $p>1$.

The second main result in \cite{GIM} involves nonlinear generalizations of the Cauchy problem for the heat equation \eqref{calore}. In particular, uniqueness issues for  nonnegative bounded solutions to \eqref{pp1} are investigated, where $\phi=\psi^{-1}$ and $\psi$ is the nonlinearity appearing in the elliptic equation \eqref{eqe-intro}, so that $\phi$ is \emph{concave}, strictly increasing (with $ \phi(0)=0 $), continuous in $[0,+\infty)$ and $  C^1$ in $(0,+\infty)$, $ T \in (0,+\infty] $ and $u_0\in L^\infty(M)$, with $ u_0 \ge 0 $. In this case, the prototype nonlinearity in \eqref{pp1} is $\phi(u)=u^m$ with $m\in(0,1)$, namely the previously mentioned \it fast diffusion equation\rm, see \cite{DUV, V}. In fact, it is shown in \cite{GIM} that stochastic completeness is equivalent to the uniqueness of nonnegative bounded solutions (in the very weak sense) to the Cauchy problem \eqref{pp1} for one, hence all, nonlinearities $\phi$ having the just stated properties. We stress that such a result yields the very first example in which nonuniqueness of bounded solutions to the fast diffusion equation with bounded data occurs, as this is the case on all stochastically incomplete manifolds. On the contrary, we observe that in \cite{GMP-tams} it is shown that uniqueness of (nonnegative, very weak) solutions to the fast diffusion equation holds even for data that are in $L^2_{\mathrm{loc}}(M)$, on a class of (stochastically complete) manifolds satisfying suitable curvature bounds from below, which are very close to those distinguishing stochastically complete from stochastically incomplete manifolds. We also refer to the seminal paper \cite{HP} for important results originally proved in the Euclidean framework, in which uniqueness is shown for (strong) solutions taking $L^1_{\mathrm{loc}}$ initial data.

\smallskip

In the present paper, we are able to  \emph{remove} the \emph{convexity} and \emph{concavity} assumptions on $ \psi $ and, respectively, $ \phi $, which were essential in \cite{GIM}. Moreover, we can strongly relax the further hypotheses on such nonlinearities: we \emph{do not} assume any \emph{regularity} of $\psi$ and $ \phi $ beyond continuity and \emph{drop strict monotonicity}. In addition, we manage to deal with solutions \it that may change sign\rm, whereas nonnegativity was a crucial requirement for the previous techniques to work.

Let us briefly explain the main technical novelties of our approach. In \cite{GIM}, a single equivalence result was stated and proved (see Theorem 1.1 there), where the function $ \psi $ was identified with $ \phi^{-1} $. The first key point towards a more general characterization consists of uncoupling \eqref{eqe-intro} and \eqref{pp1}, thus establishing two corresponding (independent) equivalence results, namely Theorems \ref{teo-ell} and \ref{teo-par} (see below). In particular, the proof of the former is now completely self contained and merely relies on comparison principles for very weak sub- and supersolutions, along with the construction of suitable nontrivial solutions. We stress that, contrarily to \cite{GIM}, we \emph{avoid} any use of \emph{Kato-type inequalities}. This strategy turns out to be successful regardless of the convexity, regularity and strict monotonicity of $ \psi $.

As for the evolutionary problem, which is by far the most delicate one,
we use a delicate barrier method that, instead of looking directly for a ``large'' solution to \eqref{pp1} in a pointwise sense (which was the original approach of \cite{GIM}), aims at constructing a solution that satisfies a prescribed condition ``on a portion of the infinity of $M$'', in an integral sense, and this is precisely what allows us to get rid of the concavity of $ \phi $. The arbitrariness of such a condition is the one ensuring the multiplicity of solutions on a stochastically incomplete manifold. The reverse implication follows a strategy similar to that of \cite{GIM}, except that concavity is exploited only at the level of \emph{moduli of continuity} (rather than directly on $ \phi $). We also stress that, by means of our general arguments, we are able to treat \emph{sign-changing solutions}, for nonlinearities that are merely continuous and not necessarily strictly increasing. Moreover, our methods only rely on local comparison principles and the \emph{weak} regularity of the constructed solutions. For these reasons, we believe that the approach we adopt here is robust enough to be extended to rougher frameworks (e.g.~operators with nonsmooth weights). On the other hand, such a generality calls for an abstract viewpoint in order to be able to show that the solutions to the local approximating problems we want to solve (on bounded smooth domains -- see Section \ref{aux}) exist and suitably converge to a global solution in $M$. As opposed to \cite{GIM}, the latter \emph{do not} enjoy any \emph{monotonicity} property with respect to the domain, so that compactness must be proved with subtle ad-hoc arguments that are not very common in the literature related to the filtration equation. To this end, it is convenient to resort, locally, to the theories of gradient flows in Hilbert spaces and $m$-accretive operators in Banach spaces (see \cite{Bre} and \cite[Chapter 10]{V}), which provide crucial a priori estimates.

\smallskip

Finally, we mention that the evolutionary problem \eqref{pp1} has been the object of investigation of recent works when $\phi$ is \emph{convex} (i.e.~the case not covered by \cite{GIM}), especially in the specific framework of the porous medium equation $\phi(u)=u^m$ with $m>1$, see \cite{GMP, GMP2, GMP-dcds, GMV, GMV2}. In these papers, the main issues addressed are existence and uniqueness/nonuniqueness of solutions as well as their asymptotic behavior. A general theorem \it characterizing \rm uniqueness, or nonuniqueness, of solutions corresponding to bounded data, in terms of precise features of the manifolds involved, was anyway missing there. Thus, the present results are entirely new even in the special, but widely investigated, case of the porous medium equation. In this regard, we point out that as a byproduct of our methods of proof we establish a property concerning the \emph{lack of compact support} of solutions (Corollary \ref{cor-pme}), which is somewhat surprising for the porous medium equation having in mind its well-known finite speed of propagation in the Euclidean setting \cite{V}. Besides that, as discussed above, the generality in which we work allows us to also deal, for the first time, with global well-posedness issues for the one-phase classical Stefan problem on manifolds and its two-phase variant: we refer the reader to \cite{ACS,CE,CF,Du,F,F2,Gu,QV} among the extensive (Euclidean) literature.

\subsection{Statements of the main results}
First of all, we introduce the following classes of real functions:
\begin{equation*}\label{class}
	\mathfrak{C} := \left\{ \phi : \mathbb{R} \to \mathbb{R} \, :  \ \text{$\phi$ is continuous, nondecreasing and nonconstant} \right\}
\end{equation*}
and
\begin{equation*}\label{class-plus}
	\mathfrak{C}^+ := \left\{ \psi \in \mathfrak{C} :  \ \psi(0)=0 \text{ and } \psi>0 \text{ in } (0,+\infty) \right\} .
\end{equation*}
For any $ \psi \in \mathfrak{C}^+ $, we will consider the \emph{nonlinear elliptic equation}
\begin{equation}\label{eqe}
	\Delta W=\psi(W) \qquad \text{in } M \, ,
\end{equation}
and for any $ \phi \in \mathfrak{C} $, $ T \in (0,+\infty] $ and $ u_0 \in L^\infty(M)$ we will also treat the \emph{nonlinear parabolic Cauchy problem}
\begin{equation}\label{pp}
	\begin{cases}
	\partial_t	u=\Delta \phi(u) & \text{in } M \times(0,T) \, , \\
		u = u_0 & \text{on } M \times \{ 0 \} \, .
	\end{cases}
\end{equation}
As mentioned above, solutions and sub/supersolutions to \eqref{eqe} or \eqref{pp} are tacitly meant in the \emph{very weak} sense (or equivalently in the sense of distributions), according to Definitions \ref{vw} or \ref{ds} below, respectively. We refer to Remark \ref{notions-sols} for the significance of such a choice.

\smallskip

Our first result concerns the equivalence between stochastic incompleteness and the existence of nonnegative, nontrivial, bounded subsolutions, or solutions, to \eqref{eqe}.

\begin{thm}\label{teo-ell}
Let $ M $ be a noncompact Riemannian manifold. Let $ \psi \in \mathfrak{C}^+ $.
Then the following properties are equivalent:
\begin{enumerate}[(a)]
\item \label{j} $ M $ is stochastically incomplete;

\item \label{jj} The nonlinear elliptic equation \eqref{eqe} admits infinitely many nonnegative, nontrivial, bounded solutions;

\item  \label{jjj}  The nonlinear elliptic equation \eqref{eqe} admits a nonnegative, nontrivial, bounded subsolution.

\end{enumerate}
\end{thm}

\begin{oss}[Assumptions on $\psi$]\rm  \label{rr}	
Let us comment on the sharpness of the  class $ \mathfrak{C}^+ $ we consider.
	\begin{itemize}
		\item \emph{Positivity} away from $0$ is necessary. In fact, if $\psi$ vanishes in the interval $(0,b)$ for some $b>0$, then any constant lying in such an interval is a nontrivial solution to \eqref{eqe}, regardless of the stochastic incompleteness of $M$.
	
		\smallskip
		
		\item The hypothesis $\psi(0)=0$ is also crucial. Indeed, if $\psi(0)>0$ and $W$ is nonnegative, nontrivial, bounded and satisfies $\Delta W\ge\psi(W)$, then a fortiori
		\begin{equation*}
		\frac{\Delta W}{\psi(0)} \ge1 \qquad \text{in } M \, .
		\end{equation*}
		Hence, the function
		$$
		W_0 := \tfrac{\left\| W \right\|_{L^\infty(M)} - W }{\psi(0)}
		$$
		is nonnegative, nontrivial, bounded and satisfies $- \Delta  W_0 \ge 1$ in $M$. As a consequence, it is straightforward to check that $M$ is nonparabolic, i.e.~it admits a positive minimal Green function $\mathcal{G}$. Moreover, the function
		$$  \gamma (x) := \int_M \mathcal{G}(x,y) \, d\V(y) $$
		 turns out to be the minimal positive solution to $-\Delta \gamma=1$, which is thus bounded above by $ W_0 $. Therefore, from \cite[Theorem 28]{PPS} (see also \cite[Theorem 3]{Grig-old}) it follows that $M$ is not \mbox{$L^1$-Liouville} (the fact that the authors treat manifolds with boundary is immaterial). However, this property is in general \emph{stronger} than stochastic incompleteness, since \cite[Example 36]{PPS} provides the construction of a manifold which is stochastically incomplete and $L^1$-Liouville.
		
		\smallskip
		
		\item Monotonicity is unavoidable in order to be able to pass from \eqref{jjj} to \eqref{jj} in Theorem \ref{teo-ell}. In fact, the construction of a bounded \emph{solution} to \eqref{eqe} that lies above a given nontrivial subsolution strongly relies on a \emph{local comparison principle} (see Subsection \ref{subs-ell}). In turn,  the validity of such a comparison principle in the generality we treat requires $ \psi $ to be nondecreasing. Nevertheless, in Theorem \ref{cor-gen-prs} we will see that this property, along with the positivity of $W$, can be almost entirely dropped as long as subsolutions are concerned.
		
	\end{itemize}
	
\end{oss}

As an immediate consequence of Theorem \ref{teo-ell}, we can infer the following.

\begin{cor}\label{one-all}
Let $ M $ be a noncompact Riemannian manifold. Assume that the nonlinear elliptic equation \eqref{eqe} admits a nonnegative, nontrivial, bounded subsolution for some $ \psi \in \mathfrak{C}^+ $. Then it admits infinitely many nonnegative, nontrivial, bounded solutions for all $ \psi \in \mathfrak{C}^+ $.
\end{cor}

With similar methods of proof, we are also able to establish a generalized version of Theorem \ref{teo-ell}, where positivity and monotonicity can be relaxed.

\begin{thm}\label{cor-gen-prs}
Let $ M $ be a noncompact Riemannian manifold. Let $ f : \mathbb{R} \to \mathbb{R} $ be a continuous function such that
\begin{equation}\label{assumpt-f}
f(w_0)=0 \qquad \text{and} \qquad f(w) > 0  \quad \forall w \in (w_0,w_1] \, ,
\end{equation}
for some $ w_1>w_0 $. Then the following properties are equivalent:
\begin{enumerate}[(a)]

\item \label{k} $ M $ is stochastically incomplete;
\setcounter{enumi}{3}
\item \label{kk} There exists $ w \in L^1_{\mathrm{loc}} (M) $ with $ {w^+} \in L^\infty(M) $ and $ f(w) \in L^1_{\mathrm{loc}} (M) $ satisfying
\begin{equation}\label{f-w}
\Delta w \ge f(w)  \qquad \text{in the sense of distributions in $M$}
\end{equation}
and
\begin{equation}\label{f-w-bis}
f(w^\ast)>0 \, , \qquad \text{where } w^*:= \underset{M}{\operatorname{ess}\sup} \ w \, .
\end{equation}

\end{enumerate}
If, in addition, $f$ is nondecreasing in $[w_0,w_1]$, then (a) and (d) are also equivalent to the following:

\begin{enumerate}[(d*)]

\item  There exist infinitely many $ w \in L^\infty (M) $ satisfying \eqref{f-w-bis} and
\begin{equation*}\label{f-w-sol}
\Delta w = f(w)  \qquad \text{ in $M$} \, .
\end{equation*}

\end{enumerate}

\end{thm}

\begin{oss}[Comparison with a result of \cite{PRS}]\label{setti}\rm
First of all, we observe that Theorem \ref{cor-gen-prs} is equivalent to asserting that $ M $ is stochastically \emph{complete} if and only if all function $ w $ as in \eqref{kk} that satisfies $ \Delta w \ge f(w)  $ is such that $ f(w^\ast) \le 0 $. Note that such equivalence holds for each \emph{fixed} nonlinearity $f$ as above. This is a stronger statement with respect to Theorem 3.1 and condition (vii) on page 43 of \cite{PRS}  (see also \cite{PPS}), in which stochastic completeness is shown to be equivalent to the fact that for \emph{all} nonlinearity $f$ and any function $ w $ such that $ \Delta w \ge f(w)  $ it holds $ f(w^*) \le 0 $. Also, we recall that the additional assumption $ w \in C^2(M) $ is required in \cite{PRS}, whereas our results basically hold under the least regularity hypotheses ensuring the meaningfulness of \eqref{f-w}; this may be especially relevant in settings where regularity is not even guaranteed for solutions, e.g.~when considering differential operators with degenerate or singular weights (see e.g.~\cite{GMPo13}).
\end{oss}

\begin{oss}[Assumptions on $f$]\label{r2}\rm
As concerns the ``local vanishing and positivity'' condition \eqref{assumpt-f}, we point out that it is actually \emph{not} needed in the proof of \eqref{kk}$ \Rightarrow $\eqref{k}. On the other hand, it is clear that there is no hope to establish the inverse implication \eqref{k}$ \Rightarrow $\eqref{kk} without further assumptions on $f$. For instance, if $ \inf f > 0 $, then an argument similar to that of Remark \ref{rr} (second item) shows that, should a function $w$ as in \eqref{kk} exist, then $M$ would not be $ L^1 $-Liouville. On the contrary, it seems an open problem to establish whether the implication \eqref{k}$ \Rightarrow $\eqref{kk} can still be proved when $ f $ has no zeros but $ \inf f = 0 $; note that this might only be possible for some function $w$ as in \eqref{kk} with $ \inf w = -\infty $, otherwise we could immediately reduce the problem to the case $ \inf f > 0 $.
\end{oss}

Our second main characterization of stochastic incompleteness is provided in terms of the multiplicity of bounded solutions to \eqref{pp} or, equivalently, of a substantial ill-posedness of the problem in the $L^\infty$ framework.

\begin{thm}\label{teo-par}
Let $ M $ be a noncompact Riemannian manifold. Let $ \phi \in \mathfrak{C} $, $ T \in (0,+\infty] $ and $ u_0 \in L^\infty(M)$. Then the following properties are equivalent:
\begin{enumerate}[(a)]
\item $ M $ is stochastically incomplete; \label{i}
\setcounter{enumi}{4}
\item \label{ii} The Cauchy problem \eqref{pp} admits infinitely many solutions in $ L^\infty(M\times(0,T)) $;
\item \label{iii} The Cauchy problem \eqref{pp} admits at least two solutions in $ L^\infty(M\times(0,T)) $.
\end{enumerate}
\end{thm}

Similarly to Corollary \ref{one-all}, from Theorem \ref{teo-par}  we can readily deduce the following.

\begin{cor}\label{one-all2}
Let $ M $ be a noncompact Riemannian manifold. Assume that the Cauchy problem \eqref{pp} admits at least two solutions in $ L^\infty(M\times(0,T)) $ for some $ \phi \in \mathfrak{C} $, some $ T \in (0,+\infty] $ and some $ u_0 \in L^\infty(M)$. Then it admits infinitely many solutions in $ L^\infty(M\times(0,T)) $ for all $ \phi \in \mathfrak{C} $, all $ T \in (0,+\infty] $ and all $ u_0 \in L^\infty(M)$.
\end{cor}

Again, we stress that there is no available result in the literature characterizing uniqueness, or nonuniqueness, of solutions to \eqref{pp} even in the specific case of the widely studied \emph{porous medium equation}, namely $\phi(u)=|u|^{m-1}u$ for $m>1$, and {a fortiori} in the broad class of nonlinearities considered here. Interestingly, as a byproduct of our techniques, we discover a phenomenon of \emph{lack of compact-support preservation}, which is quite relevant especially in the porous-medium case, marking a striking difference with respect to the Euclidean setting.

\begin{cor}\label{cor-pme}
	Let $ M $ be a stochastically incomplete noncompact Riemannian manifold. Let $ \phi \in \mathfrak{C} $, $ T \in (0,+\infty] $ and $ u_0 \in L^\infty(M)$. Then there exist infinitely many solutions $ u \in L^\infty(M\times(0,T)) $ to the Cauchy problem \eqref{pp} such that
	\begin{equation}\label{lack-supp}
		\operatorname{supp} u \rfloor_{M \times [ s, t ]}  \not \Subset M \times [ s, t ] \qquad \, \forall t,s : \ 0 < s < t < T \, ,
	\end{equation}
	where $ \operatorname{supp} u $ is understood in the sense of measures.
\end{cor}

\begin{oss}[Stochastic \emph{completeness}]\rm
	Clearly, both Theorem \ref{teo-ell} and Theorem \ref{teo-par} may equivalently be stated in terms of stochastic \emph{completeness}. This means that, under the running assumptions on $ \psi $ and $ \phi $, on a stochastically complete manifold $ M $ the only nonnegative bounded solution to \eqref{eqe} is the constant $ W = 0 $, whereas problem \eqref{pp} is always well posed in the $ L^\infty(M \times (0,T)) $ setting.
\end{oss}

\subsection{Notation and basic definitions}\label{preliminaries}

We collect here the essential definitions of solutions to \eqref{eqe} and \eqref{pp}, as well as of sub/supersolutions, with which we will work. Before, let us introduce some basic notation and terminology that will be regularly adopted in the sequel.

\smallskip

We denote by $\V$ the Riemannian volume measure on $M$, and by $\s$ the corresponding $(N-1)$-dimensional Hausdorff measure. From here on, we will take for granted that $ M $ is a noncompact Riemannian manifold of dimension $ N \ge 2 $. For simplicity, we will also tacitly assume that $M$ is \emph{connected}, but our arguments can be easily extended to deal with the non-connected case too (see Remark \ref{non-conn} at the end of the paper). Except when it is necessary due to possible ambiguities, in order to lighten notation we will drop the explicit writing of ``a.e.''~for properties, identities or inequalities that hold almost everywhere with respect to $ \V $. As we will see in the next section, we will often deal with problems posed on sufficiently regular subsets $ E \Subset M $; for a function $f$ (or a measure) defined in $ M $ or in a set that contains $E$, we will write $ f \rfloor_E $ whenever we want to emphasize that we specifically refer to its restriction on $E$. Also, we adopt the standard notations $ H^1_0(E) $ and $ H^{-1}(E) $ to refer to the closure of $ C^\infty_c(E) $ with respect to the norm $ \| \nabla (\cdot) \|_2 $  and, respectively, its dual space. Similarly, we denote by $ H^1(E) $ the Sobolev space of $L^2(E)$ functions whose distributional gradient is also in $L^2(E)$. Clearly, if $E$ is regular enough (which will always be the case for us), the latter coincides with the closure of $ C_c^\infty\!\left( \overline{E} \right) $ with respect to the norm $  \| \cdot \|_2 +  \| \nabla (\cdot) \|_2 $.

\smallskip

Now, let us provide the following standard definition of a \emph{very weak} solution (or solution in the sense of distributions) to \eqref{eqe}. Although our main interest is for nonlinearities in $ \mathfrak{C}^+ $, for technical reasons it is convenient to temporarily work in the wider class $ \mathfrak{C} $.

\begin{den}\label{vw}
Let $ \psi \in \mathfrak{C} $. We say that a function $ W \in L^1_{\mathrm{loc}}(M) $ is a very weak solution to the nonlinear elliptic equation \eqref{eqe} if $ \psi(W) \in L^1_{\mathrm{loc}}(M) $ and it satisfies
\begin{equation}\label{def-sol-ell}
 \int_M W \, \Delta \eta  \, d\V = \int_M \psi(W) \, \eta  \, d\V
\end{equation}
for all $ \eta \in C^2_c( M ) $. Similarly, we say that a function $ W \in L^1_{\mathrm{loc}}(M) $ is a very weak subsolution (resp.~supersolution) to \eqref{eqe} if $ \psi(W) \in L^1_{\mathrm{loc}}(M) $ and \eqref{def-sol-ell} is satisfied with ``='' replaced by ``$\ge$'' (resp.~``$ \le $''), for all nonnegative $ \eta \in C^2_c( M  ) $.
\end{den}

As concerns \eqref{pp}, we can also give the usual definition of a very weak solution in which the initial datum is absorbed into the formulation itself, rather than taken with continuity.

\begin{den}\label{ds}
Let $ \phi \in \mathfrak{C} $, $ T \in (0,+\infty] $ and $ u_0 \in L^\infty(M)$. We say that a function $ u \in L^\infty(M \times (0,T)) $ is a very weak solution to the Cauchy problem \eqref{pp} if it satisfies
\begin{equation}\label{def-sol}
\int_0^T \int_M u \, \partial_t \xi  \, d\V dt + \int_0^T \int_M \phi(u) \, \Delta \xi \, d\V dt = - \int_M u_0(x) \, \xi(x,0) \, d\V(x)
\end{equation}
for all $ \xi \in C^2_c( M \times [0,T) ) $.
\end{den}

\begin{oss}[On the notions of solution] \rm \label{notions-sols}
	Both for elliptic and parabolic problems, we decided to work with the concept of \emph{very weak} solution (or sub/supersolution), which as the name suggests is to some extent the weakest notion one can employ to give a meaning to \eqref{eqe} or \eqref{pp}. A strictly related concept is that of \emph{weak solution}, namely when $W$ (resp.~$ \phi(u) $) belongs at least locally to the Sobolev space $ H^1 $, so that it is feasible to use it as a test function in \eqref{def-sol-ell} (resp.~\eqref{def-sol}). Because subsolutions to \eqref{eqe} are in particular \emph{subharmonic}, the very recent result \cite[Theorem 3.1]{PVV} ensures that they are actually weak. Nevertheless, as mentioned above, in order to keep the discussion as much as possible self-contained and open to other frameworks where such a property is unknown, we preferred not to take advantage of this fact in the proof of Theorem \ref{teo-ell}. On the contrary, to our knowledge there is no general result establishing that any given very weak solution to \eqref{pp} is necessarily weak, especially in the cases when uniqueness fails (e.g.~stochastically incomplete manifolds).
\end{oss}

In the following, in order to prove some auxiliary intermediate results, we will focus on suitable elliptic and parabolic problems posed on \emph{bounded smooth domains}. More precisely, we say that $ E \subset M $ is a bounded smooth (or regular) domain if it is a precompact, open and connected set whose boundary $ \partial E $ is an $ (N-1) $-dimensional submanifold. In this case, $ \partial E $  is naturally orientable, and it is thus well defined an outward (smooth) normal vector field that we denote by $ \nu $. We point out that, when dealing with $ \partial E $, the reference measure is implicitly assumed to be $\s$. A connected noncompact Riemannian manifold $ M $ always admits a \emph{regular exhaustion}, namely a sequence of bounded smooth domains $ \{ E_k \} \subset \mathcal{P}(M) $ such that $ E_{k} \Subset E_{k+1} $ and $ \bigcup_k E_k = M $. For further references in this regard, see e.g.~\cite[Subsection 2.1]{GMP-tams} or the monograph \cite{Lee}. Note that if $M$ is complete and there exists $ o \in M $ with empty cut locus (or even empty away from a compact set), one can simply choose $ E_k = B_k(o) $. In order to avoid redundancy, when we say ``regular exhaustion'' we implicitly mean a  \emph{regular exhaustion of $M$}.

\smallskip

The paper is organized as follows. In Section \ref{aux} we collect some crucial local comparison principles for very weak solutions, subsolutions and supersolutions to localized versions of \eqref{eqe} and \eqref{pp}, along with related fundamental existence results. The corresponding proofs are quite technical because of the distributional sense in which the problems are understood and the very general nonlinearities we treat. For the sake of readability, they have been deferred to Section \ref{comparison-ell} (elliptic) and Section \ref{comparison-par} (parabolic). The proofs of all our main results are contained in Section \ref{proof}, and basically they can be read independently of Sections \ref{comparison-ell} and \ref{comparison-par}.

\section{Auxiliary existence and comparison results for local problems}\label{aux}

This section is devoted to establishing some preliminary results that are key to the proofs of Theorems \ref{teo-ell}, \ref{cor-gen-prs} and \ref{teo-par}. More precisely, we will solve the analogues of \eqref{eqe} and \eqref{pp} on bounded smooth domains, completed with suitable boundary conditions, and show related comparison principles. We stress that establishing such properties in a \emph{very weak} setting is more involved, and requires several additional technical tools. Among them, for the reader's convenience we recall an important result from \cite{GMP-tams}, whose statement is adapted to the (simpler) version that is enough to our purposes.

\begin{pro}[Proposition 3.1 of \cite{GMP-tams}]\label{tams-prop}
Let $ M $ be a connected and noncompact Riemannian manifold.  Let $ T \in (0,+\infty] $ and $ g \in L^1_{\mathrm{loc}}(M \times [0,T)) $ satisfy
\begin{equation*}\label{p1}
\int_0^T \int_{M} g \, \Delta \xi \, d\V dt \le \mathsf{F}(\xi,\partial_t \xi)
\end{equation*}
for all nonnegative $ \xi \in C^2_c( M \times [0,T) ) $, where $ \mathsf{F} $ is a continuous functional on $ C_c(M \times [0,T)) \times  C_c(M \times [0,T)) $. Then there exists a regular exhaustion $ \{ E_k \} \subset   \mathcal{P}(M) $, possibly depending on $g$, such that
\begin{equation*}\label{p2}
\int_0^T \int_{E_k} g \, \Delta \xi \, d\V dt - \int_0^T \int_{\partial E_k} g \rfloor_{\partial E_k} \, \frac{\partial \xi}{\partial \nu} \, d\s dt \le \mathsf{F}\!\left(\overline{\xi},\partial_t \overline{\xi} \right)
\end{equation*}
for all nonnegative $\xi \in C^2_c\!\left(\overline{E}_k \times [0, T)\right)$ with $ \xi = 0 $ on $\partial E_k \times (0,T)$, for every $k \in \mathbb{N}$, where $ \overline{\xi} $ is the extension of $ \xi $ to $ M \times [0,T) $, set to be zero outside $ \overline{E}_k \times [0,T) $.
\end{pro}

We recall that the restriction $g \rfloor_{\partial E_k}$ enjoys all the properties that a ``trace'' is expected to have: in particular, and most importantly to our goals, if $ g \ge c $ (resp.~$ \le c $) holds $\V$-almost everywhere in $M$ then also $ g \rfloor_{\partial E_k} \ge c $ (resp.~$ \le c $) holds $ \s $-almost everywhere on $ \partial E_k $.

\subsection{Elliptic problems}\label{subs-ell}
We consider the Dirichlet problem
\begin{equation}\label{eq54-nonlin}
\begin{cases}
\Delta W = \psi(W) & \text{in } E \, , \\
W = \mathsf{c} & \text{on } \partial E \, ,
\end{cases}
\end{equation}
for a bounded smooth domain $ E $, a nonlinearity $\psi \in \mathfrak{C}$ and a given constant $ \mathsf{c} \in \mathbb{R} $. For simplicity, we take the boundary datum in \eqref{eq54-nonlin} to be a constant (which suffices for our goals), but of course more general traces on $\partial E$ could be prescribed.

\smallskip

Let us provide an appropriate notion of very weak (sub-, super-) solution to \eqref{eq54-nonlin}.

\begin{den}\label{very-nonlin-weak}
Let $ E \subset M $ be a bounded smooth domain. Let $ \psi \in \mathfrak{C} $ and $ \mathsf{c} \in \mathbb{R} $. We say that a function $ W \in L^1(E) $ is a very weak solution to the Dirichlet problem \eqref{eq54-nonlin} if $ \psi(W) \in L^1(E) $ and it satisfies
\begin{equation}\label{eq550}
\int_E W \, \Delta \eta \, d\V - \int_{\partial E} \mathsf{c} \, \frac{\partial \eta}{\partial \nu} \, d\s = \int_E \psi(W) \, \eta \, d\V
\end{equation}
for all $\eta\in C^2\!\left(\overline E\right)$ with $\eta=0$ on $\partial E$. Similarly, we say that a function $ W \in L^1(E) $ is a very weak subsolution (resp.~supersolution) to \eqref{eq54-nonlin} if $ \psi(W) \in L^1(E) $ and \eqref{eq550} is satisfied with ``='' replaced by ``$\ge$'' (resp.~``$ \le $''), for all nonnegative $ \eta$ as above.
\end{den}

We now state two results regarding \eqref{eqe} and \eqref{eq54-nonlin} that will be crucial in the proofs of Theorem \ref{teo-ell}, Theorem \ref{cor-gen-prs} and, partially, of Theorem \ref{teo-par}. They will be proved in Section \ref{comparison-ell}.

\begin{pro}\label{lemma-v-weak}
Let $ E \subset M $ be a bounded smooth domain. Let $ \psi \in \mathfrak{C} $ and  $ \mathsf{c} \in \mathbb{R}  $. Let $ \underline{W}$ and $\overline{W} $ be a very weak subsolution and, respectively, a very weak supersolution to the Dirichlet problem \eqref{eq54-nonlin}, in the sense of Definition \ref{very-nonlin-weak}. Then $ \underline{W} \le  \overline{W} $ in $E$.
\end{pro}

\begin{pro}\label{from-sub-to-sol}
Let $ \psi \in \mathfrak{C} $ with $ \psi(0) = 0 $. Assume that the nonlinear elliptic equation \eqref{eqe} admits a  very weak subsolution $ \underline{W} $, in the sense of Definition \ref{vw}, such that $ \underline{W}^+ \in L^\infty(M) $. Then it also admits a solution $ W $, in the sense of Definition \ref{vw}, such that
\begin{equation}\label{new-bound}
 \underline{W}^+ \le W \le \left\|  \underline{W}^+ \right\|_{L^\infty(M)}  \qquad \text{in } M \, .
\end{equation}
\end{pro}

We recall that, in the proof of Proposition \ref{from-sub-to-sol} and more generally throughout the paper, we do not take advantage of \emph{Kato-type} inequalities. As an alternative approach, one may notice that $ \underline{W}^+ $ satisfies
\begin{equation}\label{subsol-BK}
\Delta \underline{W}^+ \ge \chi_{\{ W>0 \}} \psi(W) = \psi\!\left( \underline{W}^+  \right) \qquad \text{in } M
\end{equation}
in the sense of distributions. Then, once a nonnegative and bounded subsolution such as $ \underline{W}^+ $ is available, a routine monotonicity argument (that is also used in our proof) yields a \emph{solution} complying with \eqref{new-bound}. However, as far as \emph{distributional subsolutions} are concerned, in order to rigorously obtain \eqref{subsol-BK} one needs to exploit a manifold version of the Br\'ezis-Kato inequality, which has been established only very recently in \cite[Proposition 4.1]{PVV}.

\subsection{Parabolic problems}\label{sub-pp}
For convenience, let us set
\begin{equation}\label{def-prim}
\Phi(u) := \int_0^u \phi(r) \, dr \, ,
\end{equation}
for $\phi\in\mathfrak C$. 
If $\phi(0)=0$, which is a relevant case as it will be clear shortly, due to the monotonicity of $\phi$ we have that $\phi(u) \ge  0$ for $u >  0$ and $ \phi(u) \le 0 $ for $ u<0 $. In particular, the function $\Phi$ is nonnegative and convex. In addition, for technical reasons, it is useful to define the following \emph{pseudo inverses} of $\phi$:
$$
\phi^{-1}_l(\rho) := \inf \left\{ u \in \mathbb{R} : \ \phi(u)=\rho \right\} \quad \text{and} \quad \phi^{-1}_r(\rho) := \sup \left\{ u \in \mathbb{R} : \ \phi(u)=\rho \right\}  \qquad \forall \rho \in \phi(\mathbb{R}) \, .
$$
Both functions are nondecreasing with $ \phi^{-1}_l \le \phi^{-1}_r $, and it is not difficult to check that $ \phi_l^{-1} $ is lower semi-continuous (possibly taking the value $-\infty$), whereas $ \phi_r^{-1} $ is upper semi-continuous (possibly taking the value $+\infty$). Clearly, if $ \phi $ is strictly increasing, these two functions are continuous and coincide with the actual inverse $ \phi^{-1} $.

\smallskip

Given a bounded smooth domain  $ E \subset M $, $ \phi \in \mathfrak{C} $, a time $T \in (0,+\infty]$, a constant $ \beta \in \phi(\mathbb{R}) $ and an initial datum $ u_0 \in L^\infty(E) $, we consider Cauchy-Dirichlet local problems of the form
\begin{equation}\label{eq52}
\begin{cases}
\partial_t u = \Delta  \phi(u)  & \text{in } E \times (0, T) \, , \\
\phi(u) = \beta & \text{on } \partial E \times (0,T) \, , \\
u=u_0 & \text{on }  E \times \{0\} \, .
\end{cases}
\end{equation}
Again, we limit ourselves to the case of \emph{constant} boundary data because it is more than enough to our purposes. However, we stress that more general boundary data could be treated in the spirit of \cite[Chapter 5]{V}, but this is out of scope for us.

\smallskip

Next, let us provide a suitable notion of very weak (sub-, super-) solution to \eqref{eq52}.

\begin{den}\label{pap}
Let $ E \subset M $ be a bounded smooth domain. Let $ \phi \in \mathfrak{C} $, $ T \in (0,+\infty] $, $ \beta \in \phi(\mathbb{R})$ and $ u_0 \in L^\infty(E) $. We say that a function $ u \in L^\infty(E \times (0,T)) $ is a very weak solution to the Cauchy-Dirichlet problem \eqref{eq52} if it satisfies
\begin{equation}\label{eq53}
\int_0^{T} \int_E u \, \partial_t \xi \, d\V dt + \int_0^{T}\int_E \phi(u) \, \Delta \xi \, d\V dt  =  - \int_E u_0(x) \, \xi(x,0)\, d\V(x)  +  \int_0^{T}\int_{\partial E} \beta \, \frac{\partial \xi }{\partial \nu} \, d\s dt
\end{equation}
for all $\xi \in C^2_c\!\left(\overline{E}\times [0, T)\right)$ with $ \xi = 0 $ on $\partial E \times (0,T)$. Similarly, we say that a function $ u \in L^\infty(E \times (0,T) ) $ is a very weak subsolution (resp.~supersolution) to \eqref{eq52} if \eqref{eq53} is satisfied with ``='' replaced by ``$\ge$'' (resp.~``$ \le $''), for all nonnegative $ \xi$ as above.
\end{den}

We start by stating an existence result for the \emph{homogeneous} problem \eqref{eq52}, along with some extra properties of the constructed solutions that will turn out to be key.

\begin{pro}\label{exiconstructed}
Let $ E \subset M $ be a bounded smooth domain. Let $ \phi \in \mathfrak{C} $ with $ \phi(0)=0 $, $\beta=0$ and $ u_0 \in L^\infty(E)$. Then:

\smallskip
\noindent (i) There exists a very weak solution $  u $, in the sense of Definition \ref{pap}, to the Cauchy-Dirichlet problem \eqref{eq52} with $T=+\infty$, such that
\begin{equation}\label{es100b}
	\left\| u \right\|_{L^\infty(E \times (0,+\infty))} \leq \left\| u_0 \right\|_{L^\infty(E)} .
\end{equation}

\smallskip
\noindent (ii) Furthermore, we have that
\begin{equation*}\label{eng-spaces}
 u \in AC_{\mathrm{loc}}\!\left([0,+\infty); H^{-1}(E) \right) \cap C\!\left([0,+\infty); L^1(E) \right) , \qquad \phi(u(\cdot,t)) \in H^1_0(E) \quad \forall t >0 \, ,
\end{equation*}
and the energy estimates
\begin{equation}\label{energy-local-integ}
	\int_0^t \left\| \phi(u(\cdot,s)) \right\|_{H^{1}_0(E)}^2 ds \le  \left\| \Phi(u_0) \right\|_{L^1(E)} \qquad \forall t>0
\end{equation}
and
\begin{equation}\label{energy-local}
\left\| \partial_t u(\cdot,t) \right\|_{H^{-1}(E)} = \left\| \phi(u(\cdot,t)) \right\|_{H^{1}_0(E)} \le \sqrt{\frac{ \left\| \Phi(u_0) \right\|_{L^1(E)}}{t}} \qquad \text{for a.e. } t>0
\end{equation}
hold.
	
	\smallskip
	\noindent (iii) If $ u_0 \le v_0 \in L^\infty(E) $ and we let $ v $ be the above constructed solution to the Cauchy-Dirichlet problem \eqref{eq52} with $u_0$ replaced by $v_0$, the following comparison holds:
	\begin{equation}\label{cp-classic}
	 u \le v  \qquad \text{in } E \times (0,+\infty) \, .
	\end{equation}
\end{pro}

\begin{oss}[The nonhomogeneous case]\rm \label{rem-nonhom}
When $ \phi(0) \neq 0 $ or $ \beta \neq 0 $, there is an easy way to reduce problem \eqref{eq52} to the homogeneous version treated in Proposition \ref{exiconstructed}. To this end, it is enough to choose any $ \alpha \in \phi^{-1}(\beta) $ and set $ \phi_\alpha(v):=\phi( v + \alpha) - \beta $, $ v_0 := u_0 - \alpha $ and apply Proposition \ref{exiconstructed} with $ (u_0 , \phi)$ replaced by $ \left( v_0 ,\phi_\alpha \right) $; as a result, one obtains a function $ v \in L^\infty(E \times (0,+\infty)) $ satisfying the very weak formulation
\begin{equation}\label{eq53-C}
	\int_0^{+\infty} \int_E v \, \partial_t \xi \, d\V dt + \int_0^{+\infty}\int_E \left[ \phi(v+\alpha) - \beta \right] \Delta \xi \, d\V dt  =  - \int_E \left[ u_0(x) - \alpha \right] \xi(x,0)\, d\V(x)
\end{equation}
for all $\xi \in C^2_c\!\left(\overline{E}\times [0, +\infty)\right)$ with $ \xi = 0 $ on $\partial E \times (0,+\infty)$. Upon letting $ u=v+\alpha $ and integrating by parts the integrals in \eqref{eq53-C} involving the constant terms $ \alpha $ and $ \beta $, it is plain to see that $u$ fulfills \eqref{eq53} (with $ T=+\infty $).
\end{oss}

The statement of Proposition \ref{exiconstructed} is certainly not a major novelty, but the mild assumptions we make on the nonlinearity $\phi$ call for a nonstandard construction that goes back to the abstract theory of gradient flows (or more generally maximal monotone operators) in Hilbert spaces and $m$-accretive operators in Banach spaces \cite{Bre-op,Bre,CP}, which must be properly combined to yield a family of solutions that have all the desired properties.

\smallskip

Upon taking advantage of Proposition \ref{exiconstructed} and its crucial estimates, we are able to show that solving (homogeneous) Cauchy-Dirichlet problems on an exhaustion of domains of $M$ always leads to a solution to \eqref{pp}, at least for a special class of initial data.

\begin{pro}\label{exi-sign-ch}
	Let $ \phi \in \mathfrak{C} $ with $ \phi(0)=0 $ and $ u_0 \in L^\infty(M)$ with compact support. Given any regular exhaustion $ \{ E_k \} \subset \mathcal{P}(M) $, let $ \{ u_k \} $ be the solutions to the homogeneous Cauchy-Dirichlet problems
\begin{equation*}\label{eq52-sign-ch}
	\begin{cases}
		\partial_t u_k = \Delta  \phi(u_k)  & \text{in } E_k \times (0, +\infty) \, , \\
		\phi(u_k) = 0 & \text{on } \partial E_k \times (0,+\infty) \, , \\
		u=u_0 \rfloor_{E_k} & \text{on }  E_k \times \{0\} \, ,
	\end{cases}
\end{equation*}
	provided by Proposition \ref{exiconstructed}. Then there exists a very weak solution $u$ to the Cauchy problem \eqref{pp} with $T=+\infty$, in the sense of Definition \ref{ds}, such that (up to a subsequence)
	\begin{equation}\label{prop-bre-2}
u_k \underset{k \to \infty}{\longrightarrow} u \quad \text{weakly$^*$ in } L^\infty(M \times (0,+\infty)) \qquad \text{and} \qquad \phi(u_k) \underset{k \to \infty}{\longrightarrow} \phi(u) \quad \text{in } L^2_{\mathrm{loc}}(M \times [0,+\infty)) \, ,
\end{equation}
where both $ u_k $ and $\phi(u_k)$ are implicitly extended to zero outside $E_k$.
\end{pro}

For very weak sub- and supersolutions to \eqref{eq52}, a crucial comparison principle holds. Its proof relies on a technical but by now classical duality argument, borrowed from \cite[Section 6]{V}.

\begin{pro}\label{pcpro2}
Let $ E \subset M $ be a bounded smooth domain. Let $ \phi \in \mathfrak{C} $, $T \in (0,+\infty]$, $ \beta \in \phi(\mathbb{R}) $ and $ u_0 \in L^\infty(E) $. Let $ \underline{u}$ and $ \overline{u} $ be a very weak subsolution and, respectively, a very weak supersolution to the Cauchy-Dirichlet problem \eqref{eq52}, in the sense of Definition \ref{pap}. Then
$$
\underline{u} \le \overline{u} \qquad \text{in } E \times (0,T) \, .
$$
\end{pro}

Note that, in the light of Proposition \ref{pcpro2} (and Remark \ref{rem-nonhom}), we can assert that the above constructed solution to \eqref{eq52} is \emph{unique} in the class of very weak solutions.

\smallskip

A further, and even more important consequence of Proposition \ref{pcpro2}, is that it is always possible to find a solution to \eqref{pp} which is smaller than any two given solutions. This will play a central role in the proof of Theorem \ref{teo-par}\eqref{iii}$\Rightarrow$\eqref{i}.

\begin{pro}\label{exi-mini}
Let $ \phi \in \mathfrak{C} $, $ T \in (0,+\infty] $ and $ u_0 \in L^\infty(M)$. Assume that the Cauchy problem \eqref{pp} admits two very weak solutions $ \hat u_1 $ and $ \hat u_2$, in the sense of Definition \ref{ds}. Then it also admits a very weak solution $u$ to the same Cauchy problem such that
\begin{equation}\label{minimal-ineq}
u \le \hat u_1 \wedge \hat  u_2 \qquad \text{in } M \times (0,T) \, .
\end{equation}
\end{pro}

As a byproduct of the method of proof of Proposition \ref{exi-mini}, we are able to guarantee that the \emph{smallest} possible solution to \eqref{pp}, among nonnegative bounded solutions, always exists.

\begin{pro}\label{exi-mini-real}
	Let $ \phi \in \mathfrak{C} $, $ T \in (0,+\infty] $ and $ u_0 \in L^\infty(M)$, with $ u_0 \ge 0 $. Then there exists a nonnegative very weak solution $ u $ to the Cauchy problem \eqref{pp}, in the sense of Definition \ref{ds}, called the minimal one, i.e.~such that
	\begin{equation*}\label{minimal-ineq-real}
		u \le v \qquad \text{in } M \times (0,T)
	\end{equation*}
	for any other nonnegative very weak solution $ v  $ to the same Cauchy problem. Moreover, $ u $ is independent of $ T $ and satisfies the upper bound
	\begin{equation*}\label{minimal-bd}
		u \le \left\| u_0 \right\|_{L^\infty(M)} \qquad \text{in } M \times (0,+\infty) \, .
	\end{equation*}
\end{pro}

\smallskip

The proofs of all the above results will be carried out in Section \ref{comparison-par}.

\section{Proofs for local elliptic problems}\label{comparison-ell}

\begin{proof}[Proof of Proposition \ref{lemma-v-weak}]
We employ a duality-based strategy. First of all, we observe that by subtracting the very weak formulations satisfied by $  \underline{W} $ and $ \overline{W} $  we obtain
\begin{equation}\label{vw1}
\int_E \left[ \left( \underline{W}-\overline{W} \right) \Delta \eta  + \left( \psi\!\left(\overline{W}\right) - \psi\!\left(\underline{W}\right) \right) \eta \right] d\V \ge 0 \, ,
\end{equation}
for all nonnegative $\eta\in C^2\!\left(\overline E\right)$ with $\eta=0$ on $\partial E$. Indeed, this follows from the definition of sub- and supersolutions given in Definition \ref{very-nonlin-weak}, upon subtracting the two inequalities satisfied by $\underline{W}$ and $\overline{W}$. By defining
\begin{equation*}\label{n4-ellipt}
	c(x):=
	\begin{cases}
		\frac{\psi\left(\overline{W}(x)\right) - \psi\left(\underline{W}(x) \right)}{\overline{W}(x) - \underline{W}(x)} & \text{if } \overline{W}(x) \neq \underline{W}(x) \, , \\
		0 & \text{if } \overline{W}(x) = \underline{W}(x) \, ,
	\end{cases}
\end{equation*}
we see that \eqref{vw1} can be written as
\begin{equation}\label{vw1-bis}
\int_E  \left( \underline{W}-\overline{W} \right) \left( \Delta \eta  - c \, \eta \right)  d\V \ge 0 \, ,
\end{equation}
for all $\eta $ as above. Since $\psi$ is nondecreasing and $ \underline{W},\overline{W},\psi(\underline{W}),\psi\!\left(\overline{W}\right) \in L^1(E) $, it is plain that $c$ complies with
\begin{equation*}\label{c-finite}
	0 \le c(x) < +\infty \qquad \text{for a.e. } x \in E \, .
\end{equation*}
Indeed, nonnegativity is clear from the monotonicity of $\psi$. Moreover, $c$ is a.e.\ finite since when $\overline{W}(x) = \underline{W}(x)$ by definition $c(x)=0$, and in addition $\phi\circ\overline{W}-\phi\circ\underline{W}\in L^1$, so that it is a.e.\ defined.
Hence, given an arbitrary $ \varepsilon>0 $, we can pick $ b_\varepsilon>0 $ so large that
\begin{equation}\label{c-finite-bis}
	\int_ { \left\{  c > b_\varepsilon \right\} } \left| \psi\!\left(\overline{W}\right) - \psi\!\left(\underline{W} \right) \right|  d\V < \varepsilon \, .
\end{equation}
Then, we let
$$
c_\varepsilon := c \wedge b_\varepsilon
$$
and select a sequence of nonnegative and bounded functions $ \{ c_n \} \subset C^\infty\!\left( \overline{E} \right) $ such that
\begin{equation}\label{Linf-star-C}
	\lim_{n \to \infty} c_n(x) = c_\varepsilon(x) \quad \text{for a.e. } x \in E \, , \qquad \sup_{n \in \mathbb{N}} \left\| c_n \right\|_{L^\infty(E)} < +\infty \, .
\end{equation}
For each $ n \in \mathbb{N} $, let us consider the (smooth) solution $ \eta_n $ to the following linear  Dirichlet problem:
\begin{equation}\label{dual-elliptic}
\begin{cases}
-\Delta \eta_n +  c_n \, \eta_n = \omega & \text{in } E \, ,\\
\eta_n = 0 & \text{on } \partial E \, ,
\end{cases}
\end{equation}
where $ \omega \in C^\infty_c(E) $, with $ \omega \ge 0 $, is arbitrary but fixed. Because $ c_n \ge 0 $, standard maximum principles ensure that $ 0 \le \eta_n \le C $, for a suitable positive constant $C$ depending on $E$ and $\omega$ but independent of $n$. We can now substitute $ \eta = \eta_n $ into \eqref{vw1-bis} using \eqref{dual-elliptic}, which yields
\begin{equation}\label{vw2}
\int_E \left( \overline{W}-\underline{W} \right) \omega \, d\V +  \int_E  \left( \underline{W}-\overline{W} \right) \left( c_n - c_\varepsilon \right) \eta_n \,  d\V +  \int_E  \left( \underline{W}-\overline{W} \right) \left( c_\varepsilon - c \right) \eta_n \,  d\V \ge 0 \, .
\end{equation}
We aim at showing that the rightmost terms on the left-hand side of \eqref{vw2} vanish as $ n \to \infty $ and $\varepsilon \to 0$. From the definitions of $c$, $ c_\varepsilon $, \eqref{c-finite-bis} and the uniform bound on $\eta_n$, we have:
\begin{equation}\label{vw3}
\begin{aligned}
\int_E  \left| \left( \underline{W}-\overline{W} \right) \left( c_\varepsilon - c \right) \eta_n \right|  d\V = & \, \int_{\left\{  c > b_\varepsilon \right\} }  \left| \left( \underline{W}-\overline{W} \right) \left( c - b_\varepsilon \right) \eta_n \right|  d\V \le \int_{\left\{  c > b_\varepsilon \right\} }  \left| \left( \underline{W}-\overline{W} \right) c \, \eta_n \right|  d\V \\
\le & \, C \, \int_ { \left\{  c > b_\varepsilon \right\} } \left| \psi\!\left(\overline{W}\right) - \psi\!\left(\underline{W} \right) \right|  d\V < C \varepsilon \, .
\end{aligned}
\end{equation}
On the other hand, from \eqref{Linf-star-C}, the uniform bound on $\eta_n$ and dominated convergence, we can infer that for every fixed $\varepsilon>0$
\begin{equation}\label{vw4}
\lim_{n \to \infty} \int_E  \left( \underline{W}-\overline{W} \right) \left( c_n - c_\varepsilon \right) \eta_n \,  d\V = 0 \, .
\end{equation}
Hence, upon letting first $ n \to \infty $ and then $\varepsilon \to 0 $ in \eqref{vw2}, using \eqref{vw4} and \eqref{vw3}, we end up with
$$
\int_E \left( \overline{W}-\underline{W} \right) \omega \, d\V \ge 0 \, .
$$
The thesis follows owing to the arbitrariness of $ \omega $.
\end{proof}

\begin{proof}[Proof of Proposition \ref{from-sub-to-sol}]
By assumption, $ \underline{W} $ satisfies
\begin{equation*}
 \int_M \underline{W} \, \Delta \eta  \, d\V \ge  \int_M \psi\!\left(\underline{W}\right) \eta  \, d\V
\end{equation*}
for all nonnegative $ \eta \in C^2_c( M ) $. Hence, as a special case of Proposition \ref{tams-prop}, we can infer that there exists a regular exhaustion $ \{ E_k \} \subset \mathcal{P}(M)  $ (possibly depending on $ \underline{W} $) such that
\begin{equation}\label{eq-sub-1}
\int_{E_k} \underline{W} \, \Delta \eta \, d\V - \int_{\partial E_k} \left\| \underline{W}^+ \right\|_{L^\infty(M)}  \frac{\partial \eta}{\partial \nu} \, d\s \ge \int_{E_k} \underline{W} \, \Delta \eta \, d\V - \int_{\partial E_k}  \underline{W} \rfloor_{\partial E_k}  \frac{\partial \eta}{\partial \nu} \, d\s  \ge \int_{E_k} \psi(\underline{W}) \, \eta \, d\V \, ,
\end{equation}
for every $ k \in \mathbb{N} $ and all nonnegative $ \eta \in C^2_c\!\left(\overline{E}_k \right) $ with $ \eta=0 $ on $ \partial E_k $, the leftmost inequality deriving from the fact that $ \frac{\partial \eta}{\partial \nu} \le 0 $ on $ \partial E_k $. Next, we consider the (weak) solutions $\{W_k\}$ to the following nonlinear Dirichlet problems:
\begin{equation}\label{eq-sub-2}
\begin{cases}
\Delta W_k = \psi(W_k) & \text{in } E_k \, , \\
W_k = \left\| \underline{W}^+ \right\|_{L^\infty(M)} & \text{on } \partial E_k \, ,
\end{cases}
\end{equation}
which can be constructed through a routine variational argument and satisfy
\begin{equation}\label{eq-sub-3-pre}
 \int_{E_k} W_k \, \Delta \eta \, d\V - \int_{\partial E_k}  \left\| \underline{W}^+ \right\|_{L^\infty(M)} \frac{\partial \eta}{\partial \nu} \, d\s = \int_{E_k} \psi(W_k) \, \eta \, d\V \,
\end{equation}
for all $ \eta \in  C^2_c\!\left(\overline{E}_k \right) $. Also, because $ \psi $ is nondecreasing and $ \psi(0) = 0 $, we have that the constants $ 0 $ and $ \left\| \underline{W}^+ \right\|_{\infty} $ are a subsolution and, respectively, a supersolution to \eqref{eq-sub-2}, thus classical maximum principles (or even Proposition \ref{lemma-v-weak}) yield
\begin{equation}\label{eq-sub-5}
0 \le	W_k \le \left\| \underline{W}^+ \right\|_{L^\infty(M)} \qquad \text{in } E_k \, .
\end{equation}
Thanks to \eqref{eq-sub-1} and \eqref{eq-sub-3-pre}, we are in position to apply Proposition \ref{lemma-v-weak}, which entails
\begin{equation}\label{eq-sub-4}
\underline{W} \le W_k  \qquad \text{in } E_k \, ,
\end{equation}
for every $ k \in \mathbb{N} $. Since $ W_{k+1} \in H^1(E_{k+1}) $, the latter has in particular a well defined trace on $ \partial E_{k} $ which is less than $ \left\| \underline{W}^+ \right\|_{\infty} $ in view of \eqref{eq-sub-5} (with $k$ replaced by $ k+1 $), therefore it is a (weak) subsolution to the Dirichlet boundary problem solved by $ W_{k} $, that is
\begin{equation*}
\begin{cases}
\Delta W_{k+1} = \psi(W_{k+1}) & \text{in } E_{k} \, , \\
W_{k+1} \le \left\| \underline{W}^+ \right\|_{L^\infty(M)} & \text{on } \partial E_{k} \, ,
\end{cases}
\end{equation*}
which entails $ W_{k+1} \le W_{k} $ in $ E_{k} $, again by virtue of Proposition \ref{lemma-v-weak}. The sequence $ \{ W_k \} $, extended to the value $ \left\| \underline{W}^+ \right\|_{\infty} $ in $ M \setminus E_k $, is then nonincreasing in $k$, therefore it admits a pointwise limit $ W $ which, upon letting $ k \to \infty $ in \eqref{eq-sub-5} and \eqref{eq-sub-4}, turns out to comply with \eqref{new-bound}. Finally, we observe that for each $ \eta \in C_c^2(M)  $ there exists a large enough $ k_0 $ such that $ \operatorname{supp} \eta \Subset E_k  $ for every $ k \ge k_0 $, whence \eqref{eq-sub-3-pre} yields
\begin{equation}\label{eq-sub-6}
\int_{E_k} W_k \, \Delta \eta \, d\V   = \int_{E_k} \psi(W_k) \, \eta \, d\V \, .
\end{equation}
Thanks to the just proved monotone convergence of $ W_k $ to $W$, and the continuity of $ \psi $, we can safely pass to the limit in \eqref{eq-sub-6} and deduce that $W$ is indeed a solution to \eqref{eqe} in the sense of Definition \ref{vw}.
\end{proof}

\section{Proofs for local parabolic problems}\label{comparison-par}

\begin{proof}[Proof of Proposition \ref{exiconstructed}] The existence of a very weak solution to the corresponding Cauchy-Dirichlet problem \eqref{eq52} when $ \beta=\phi(0)=0 $ can be obtained as a direct consequence of the theory of gradient flows in Hilbert spaces, applied to evolutionary PDEs, developed in \cite{Bre} (see also \cite[Section 10.1]{V}). Indeed, by virtue of \cite[Corollary 31]{Bre} (for completeness see also Theorems 17 and 23 of the same paper), for homogeneous boundary data a solution to problem \eqref{eq52} can be obtained as the gradient flow in $ H^{-1}(E) $ of the energy functional
$$
J(u) :=
\begin{cases}
\int_{E} \Phi(u) \, d\V & \text{if } u \in H^{-1}(E) \cap L^1(E) \, , \\
+ \infty & \text{elsewhere} \, ,
\end{cases}
$$
where $ \Phi $ has been defined in \eqref{def-prim}. We recall that a function $u\in L^1(E)$ belongs to  $H^{-1}(E)  $ if and only if the map
$$
\eta \mapsto \int_E u \, \eta \, d\V \qquad \forall \eta \in C^\infty_c(E)
$$
can be extended to become a continuous functional on $H^1_0(E)$, see e.g.\ \cite[Example 3]{Bre}. Note that the function $ \phi $, actually called $\beta$ in \cite{Bre}, is not required to be strictly monotone (in fact also continuity could be dropped provided it is a maximal graph), but it must satisfy the condition $ \lim_{u \to \pm \infty} \phi(u) = \pm \infty $ for coercivity purposes. Nevertheless, since we are interested in bounded solutions only, we will explain below why this is not restrictive for us. Hence, such an abstract construction provides a solution to \eqref{eq52} in the sense that $u$ is a curve in $ C\!\left([0,+\infty);H^{-1}(E)\right)$ satisfying $ u(0)=u_0 $,
$$ u \in L^\infty_{\mathrm{loc}}\!\left((0,+\infty);L^1(E)\right) , \quad \partial_t u \in L^\infty_{\mathrm{loc}}\!\left((0,+\infty);H^{-1}(E)\right)  , \quad  \phi(u) \in L^\infty_{\mathrm{loc}}\!\left((0,+\infty);H^{1}_0(E)\right) , $$
along with the key identity
\begin{equation}\label{H-1-bre}
_{H^{-1}(E)^{\phantom{a^a}}} \! \!\!\!\!\left\langle\partial_t  u(\cdot,t) , \eta \right\rangle_{H^1_0(E)}  = - \int_{E} \nabla \phi(u(\cdot,t)) \cdot \nabla \eta \, d\V  \qquad \text{for a.e. } t >0 \, ,
\end{equation}
for all $ \eta \in H^1_0(E) $. Moreover precisely, the identity in \eqref{energy-local} is a direct consequence of \eqref{H-1-bre}, whereas estimate \eqref{energy-local-integ} follows from the fact that the energy function $ t \mapsto J(u(\cdot,t))  $ is nonnegative, absolutely continuous on $ (0,+\infty) $, continuous on $ [0,+\infty) $ and its derivative fulfills
\begin{equation}\label{H-1-bre-bis}
\frac{d}{dt} J(u(\cdot,t))  = - \int_E \left| \nabla \phi(u(\cdot,t)) \right|^2 d\V \qquad \text{for a.e. } t >0 \, .
\end{equation}
Note that \eqref{H-1-bre-bis} is formally due to the chain rule $ \frac{d}{dt} \Phi(u) = \partial_t u \, \phi(u) $ combined with \eqref{H-1-bre}, but it can be rigorously justified owing to \cite[Theorem 3.6]{Bre-op}. As for the rightmost estimate in \eqref{energy-local}, it is due to the fact that $ \phi(u(\cdot,t)) \in H^1_0(E) $ for every $ t>0 $ and the ``slope'' function
$$
t \mapsto \left\| \phi(u(\cdot,t)) \right\|_{H^1_0(E)}
$$
is nonincreasing. Again, this can be formally seen by using $ \eta = \frac  1  2 \, \partial_t u \, \phi'(u) $ as a test function in \eqref{H-1-bre} (provided $\phi$ is differentiable), and rigorously justified according to \cite[Lemma 3.1 and Theorem 3.7]{Bre-op}. Upon exploiting such a monotonicity property in \eqref{energy-local-integ}, the claimed estimate readily follows. In particular, because $ J(u_0)= \|\Phi(u_0) \|_{L^1(E)} < +\infty $ and $ t \mapsto {1}/{\sqrt t} $ is locally integrable, we can also deduce that $ u \in AC_{\mathrm{loc}}\!\left([0,+\infty); H^{-1}(E) \right) $.

The fact that $ u $ satisfies the very weak formulation \eqref{eq53} is due to \eqref{H-1-bre}. Indeed, if we take an arbitrary $ \xi \in H^1_{\mathrm{loc}}\!\left((0,+\infty);H^{1}_0(E)\right)  $, upon testing \eqref{H-1-bre} with $ \eta = \xi(\cdot,t) $ and integrating in time we end up with the formula
\begin{equation*}\label{H-2-bre}
\begin{aligned}
& \, \int_{t_1}^{t_2} \int_E u \, \partial_t \xi \, d\V dt - \int_{t_1}^{t_2} \int_{E} \nabla \phi(u(\cdot,t)) \cdot \nabla \xi \, d\V dt \\
= & \, \int_E u(x,t_2) \, \xi(x,t_2)\, d\V(x) - \int_E u(x,t_1) \, \xi(x,t_1)\, d\V(x) \, ,
\end{aligned}
\end{equation*}
for every $t_2>t_1>0$. Also, from \eqref{energy-local-integ} we have that $  \phi(u) \in L^2\!\left((0,+\infty);H^{1}_0(E)\right) $, thus it is possible to send $ t_1 \to 0 $ in \eqref{H-1-bre} to obtain
\begin{equation}\label{h1-to-very-weak}
\int_0^{+\infty} \int_E u \, \partial_t \xi \, d\V dt - \int_0^{+\infty}\int_E \nabla \phi(u(\cdot,t)) \cdot \nabla \xi \, d\V dt  =  - \int_E u_0(x) \, \xi(x,0)\, d\V(x)
\end{equation}
for all $ \xi \in H^1\!\left((0,+\infty);H^{1}_0(E)\right)  $ that are identically zero eventually (in time). Clearly, any $ \xi \in C^2_c\!\left(\overline{E}\times [0, +\infty)\right)$ with $ \xi = 0 $ on $\partial E \times (0,+\infty)$ falls within this class, so that an integration by parts of the right term on the left-hand side of \eqref{h1-to-very-weak} shows that $ u $ does comply with \eqref{eq53}.

What we are missing is the proof of the uniform bound \eqref{es100b}, the continuity property in $ L^1(E) $ and the comparison principle \eqref{cp-classic}. To these aim, it is useful to change viewpoint and interpret $u$ as the (locally uniform) limit in $ C\!\left([0,+\infty); L^1(E)\right) $, as $ h \to 0^+, $ of the $ L^1(E) $ curves obtained by recursively solving the elliptic problems
\begin{equation}\label{ee-disc}
\begin{cases}
- \Delta \phi(u_{k+1}) + \frac{u_{k+1}}{h} = \frac{u_k}{h} & \text{in } E \, , \\
\phi(u_{k+1})=0 & \text{on } \partial E \, ,
\end{cases}
\end{equation}
and converting the discrete sequence $\{ u_k \}$ into a piecewise-constant function $ u_{(h)} $, on time intervals of size $ h $. In this way one can obtain solutions to \eqref{eq52} by resorting to the theory of $m$-accretive operators in Banach spaces (see \cite[Sections 10.2 and 10.3]{V} -- the reference Banach space here is $L^1(E)$); on the other hand, for the kind of initial data we consider, problem \eqref{ee-disc} coincides with the resolvent equation (Yosida approximation) exploited in \cite{Bre-op, Bre} to construct the above $H^{-1}(E)$ gradient flow, hence the limit object is exactly the same. The advantage is that these solutions satisfy the bounds
\begin{equation}\label{ee-disc-contr}
\left\| u_{(h)} \right\|_{L^\infty(E \times (0,+\infty))} \le \left\| u_0 \right\|_{L^\infty(M)}=:\ell \qquad \forall h >0 \, ,
\end{equation}
as a consequence of standard $ L^p(E) $ non-expansivity results for \eqref{ee-disc} (we refer e.g.~\cite[Propositions 1 and 2]{CP}), so the same bound holds for $u$ upon letting $h \to 0^+$. In particular, we observe that the behavior of $ \phi $ in the complement of the interval $[-\ell,\ell] $ is immaterial (provided monotonicity is preserved), hence one can artificially modify it there in such a way that the coercivity condition $  \lim_{u \to \pm \infty} \phi(u) = \pm \infty $ is satisfied. Likewise, the resolvent equation also enjoys the order-preserving property (see again \cite[Propositions 1 and 2]{CP}), which implies that if one takes another initial datum $ v_0 \ge u_0 $ and generates the corresponding sequence $ \{ v_k \} $  by recursively solving \eqref{ee-disc} with $ v_0 $ in the place of $u_0$, the latter is consistently pointwise ordered with respect to $ \{ u_k\} $, yielding
\begin{equation*}\label{ee-disc-contr-ordered}
u_{(h)} \le v_{(h)}  \qquad \text{in } E \times (0,+\infty) \, ,  \  \forall h >0 \, ,
\end{equation*}
with the obvious definition of $ v_{(h)} $. The comparison principle \eqref{cp-classic} then follows upon letting $h \to 0^+$. We stress that, for the validity of \eqref{ee-disc-contr}, the assumption $ \phi(0)=0 $ is crucial. At least in the case where $ \phi $ is differentiable and strictly increasing, this can be easily checked from an elementary proof, which can be obtained by using $ |u_k|^{p-1} u_{k} $ as test functions in the discrete resolvent problems (eventually letting $ p \to \infty $).
\end{proof}

The next auxiliary result will be very useful in the proof of Proposition \ref{exi-sign-ch}, in order to identify limits when we let $k \to \infty$, especially if $ \phi $ is not strictly increasing.

\begin{lem}\label{nonlin-id}
Let $ E \subset M $ be a bounded smooth domain and $ \phi \in \mathfrak{C} $. Let $ \{ u_k \} \subset L^\infty(E)  $ be a sequence such that
\begin{equation*}\label{2b-conv}
u_k \underset{k \to \infty}{\longrightarrow} u \quad \text{weakly$^\ast$ in } L^\infty(E) \qquad  \text{and} \qquad \phi(u_k) \underset{k \to \infty}{\longrightarrow} \rho \quad \text{pointwise a.e.~in } E \, .
\end{equation*}
Then $ \rho = \phi(u) $ a.e.~in $ E $.
\end{lem}
\begin{proof}
First of all we notice that, from the definitions of $ \phi^{-1}_l $ and $ \phi^{-1}_r $, we have
\begin{equation}\label{lr-inv}
\phi_l^{-1}\!\left( \phi(u_k(x)) \right) \le u_k(x)  \le \phi_r^{-1}\!\left( \phi(u_k(x)) \right) \qquad \text{for a.e. } x \in E \, ,
\end{equation}
for every $ k \in \mathbb{N} $. Thanks to the lower semi-continuity of $ \phi_l^{-1} $ and the upper semi-continuity of $ \phi_r^{-1} $, we readily infer that
\begin{equation}\label{lr-inv-2}
\phi_l^{-1}\!\left( \rho(x) \right) \le \liminf_{k \to \infty} \phi_l^{-1}\!\left( \phi(u_k(x)) \right)  \le \limsup_{k \to \infty} \phi_r^{-1}\!\left( \phi(u_k(x)) \right) \le \phi_r^{-1}\!\left( \rho(x) \right)  \qquad \text{for a.e. } x \in E \, .
\end{equation}
 Our goal is to show that
\begin{equation}\label{lr-inv-final}
	\phi_l^{-1}\!\left( \rho(x) \right) \le u(x)  \le \phi_r^{-1}\!\left( \rho(x) \right) \qquad \text{for a.e. } x \in E \,  ,
\end{equation}
whence the thesis readily follows. If we set $ \mathsf{M}:= \sup_{k \in \mathbb{N}} \| u_k \|_{\infty} $, in view of the assumptions it is plain that
$$
\left\| u \right\|_{L^\infty(E)} \le \mathsf{M}  \qquad \text{and} \qquad \left\| \rho \right\|_{L^\infty(E)} \le \left|\phi\!\left(\mathsf{M} \right) \right| \vee \left|\phi\!\left(-\mathsf{M} \right) \right| .
$$
Therefore, in order to establish \eqref{lr-inv-final}, it is enough to prove both
\begin{equation}\label{lr-inv-final_L}
	\phi_l^{-1}\!\left( \rho(x) \right) \le u(x) \qquad \text{for a.e. } x \in  A_n
\end{equation}
and
\begin{equation}\label{lr-inv-final_R}
u(x)  \le \phi_r^{-1}\!\left( \rho(x) \right) \qquad \text{for a.e. } x \in  B_n \,  ,
\end{equation}
where, for each $ n \in \mathbb{N} $ (arbitrarily large), we set
$$
A_n := \left\{ x \in E : \ \rho(x) > \phi\!\left(-\mathsf{M} \right) + \tfrac{1}{n} \right\} , \qquad B_n := \left\{ x \in E : \ \rho(x) < \phi\!\left(\mathsf{M} \right)- \tfrac{1}{n} \right\} .
$$
We will show \eqref{lr-inv-final_L} only, the proof of \eqref{lr-inv-final_R} being completely analogous. By Egorov's theorem, for every $ \varepsilon>0 $ there exists a measurable set $ \mathcal{A}_\varepsilon^n \subset A_n $ with $ \V(A_n \setminus \mathcal{A}_\varepsilon^n) < \varepsilon $, such that $ \left\{ \phi(u_k) \right\} $ converges uniformly to $ \rho $ in $ \mathcal{A}_\varepsilon^n$. In particular, we can assert that
\begin{equation*}\label{lr-inv-semi}
-\infty<\phi_l^{-1}\!\left( \phi\!\left(-\mathsf{M} \right) + \tfrac{1}{2n} \right)  < \phi_l^{-1}\!\left( \phi(u_k(x)) \right)  \qquad \forall x \in  \mathcal{A}_\varepsilon^n \, ,
\end{equation*}
for every $k$ large enough. Hence, given any measurable set $ \Omega \subset \mathcal{A}_\varepsilon^n $, thanks to  \eqref{lr-inv}, \eqref{lr-inv-2}, Fatou's lemma and the weak$^*$ convergence of $\{ u_k \}$, we obtain
$$
\int_{\Omega} \phi_l^{-1}\!\left( \rho \right) d\V \le \int_{\Omega}  \liminf_{k \to \infty} \phi_l^{-1}\!\left( \phi(u_k ) \right) d\V \le \liminf_{k \to \infty} \int_{\Omega} \phi_l^{-1}\!\left( \phi(u_k ) \right) d\V \le \lim_{k \to \infty} \int_{\Omega} u_k \, d\V = \int_\Omega u \, d\V \, ,
$$
which yields \eqref{lr-inv-final_L} by virtue of the arbitrariness of $\Omega$ and $ \varepsilon $.
\end{proof}

\begin{proof}[Proof of Proposition \ref{exi-sign-ch}]
Estimate \eqref{es100b}, applied with $E$ replaced by $E_k$ and $u$ replaced by $u_k$, reads
\begin{equation}\label{es100b-k}
	\left\| u_k \right\|_{L^\infty(E_k \times (0,+\infty))} \leq \left\| u_0 \right\|_{L^\infty(M)} ,
\end{equation}
for every $k \in \mathbb{N}$. Hence, the sequence $ \{ u_k \} $ (extended to zero outside $ E_k $) is bounded in $ L^\infty(M \times (0,+\infty)) $, therefore it admits a subsequence that converges to some function $u$ weakly$^\ast$ in $ L^\infty(M \times (0,+\infty)) $. Because $u_0$ is bounded and compactly supported, it is plain that
$$
\left\| \Phi(u_0) \right\|_{L^1(M)} < +\infty \, ;
$$
as a result, from Proposition \ref{exiconstructed}(ii) we infer the uniform estimates
\begin{equation}\label{energy-local-k}
\left\| \partial_t u_k(\cdot,t) \right\|_{H^{-1}(E_k)} = \left\| \phi(u_k(\cdot,t)) \right\|_{H^{1}_0(E_k)} \le \sqrt{\frac{ \left\| \Phi(u_0) \right\|_{L^1(M)}}{t}} \qquad \text{for a.e. } t>0 \, ,
\end{equation}
for every $k \in \mathbb{N}$, where the rightmost inequality is in fact true for every $t>0$. In particular, given any fixed $ i \in \mathbb{N} $, we have that for every $ \tau> 0 $ the sequence $ \{ u_k \} $ is (eventually) uniformly equicontinuous in $ C\!\left([0,\tau];H^{-1}(E_i) \right) $ and bounded in $ L^\infty\!\left( (0,\tau) ; L^2(E_i) \right) $. Therefore, since $ L^2(E_i) $ is compactly embedded in $ H^{-1}(E_i) $, we are in position to apply the infinite-dimensional version of the Ascoli-Arzel\`a theorem, which ensures that there exists a subsequence (that we will not relabel) such that
\begin{equation}\label{conv-loc-k}
u_k \underset{k \to \infty}{\longrightarrow} u \ \ \text{in } C\!\left([0,\tau];H^{-1}(E_i) \right) \qquad \text{and} \qquad u_k(\cdot, t) \underset{k \to \infty}{\longrightarrow} u(\cdot, t) \ \ \text{weakly in  $L^2(E_i) \, , \ \forall t \in [0,\tau] \, ,  $}
\end{equation}
where we have implicitly assumed that we are working on a subsequence where the above weak$^*$ $ L^\infty $ convergence already occurs. Now we observe that, still by virtue of \eqref{es100b-k} and \eqref{energy-local-k}, for every $t>0$ the sequence $ \left\{ \phi(u_k(\cdot,t)) \right\} $ is bounded in $ H^1(E_i) $, thus it admits a subsequence $ \left\{ \phi(u_{k_j}(\cdot,t)) \right\}_j $ that converges strongly in $ L^2(E_i) $ to some function $ \rho_t \in L^2(E_i) $. Thanks to the weak convergence in \eqref{conv-loc-k} (which is also weak$^ * $ in $ L^\infty(E_i) $), we are in position to apply Lemma \ref{nonlin-id}, ensuring that in fact $ \rho_t = \phi(u(\cdot,t)) $. Because the limit has been identified independently of the selected $j$-subsequence, we can assert that
\begin{equation}\label{conv-loc-k-bis}
\phi(u_k(\cdot,t)) \underset{k \to \infty}{\longrightarrow} \phi(u(\cdot,t)) \qquad \text{in  $L^2(E_i) \, , \  \forall t \in (0,\tau] \, .  $}
\end{equation}
As $ \left\{ \phi(u_k) \right\} $ is also bounded in $ L^\infty\!\left( (0,\tau) ; L^2(E_i) \right) $, from \eqref{conv-loc-k-bis} we easily deduce that
\begin{equation}\label{conv-loc-k-ter}
\phi(u_k) \underset{k \to \infty}{\longrightarrow} \phi(u) \qquad \text{in } L^2(E_i \times (0,\tau) ) \, .
\end{equation}
In the light of \eqref{conv-loc-k-ter}, by means of a standard diagonal procedure we can select a further subsequence such that \eqref{prop-bre-2} is fulfilled. We are only left with showing that $ u $ is a very weak solution to \eqref{pp}. To this aim, it is enough to take an arbitrary $ \xi \in C^2_c( M \times [0,+\infty) ) $, write the very weak formulation satisfied by each $u_k$ (with $k$ so large that $E_k \times [0,+\infty)$ contains the support of $ \xi $), that is
$$
\int_0^{+\infty} \int_{E_k} u_k \, \partial_t \xi \, d\V dt + \int_0^{+\infty}\int_{E_k} \phi(u_k) \, \Delta \xi \, d\V dt  =  - \int_{E_k} u_0(x) \, \xi(x,0)\, d\V(x) \, ,
$$
and pass to the limit as $k \to \infty$ upon taking advantage of \eqref{prop-bre-2}.
\end{proof}

\begin{proof}[Proof of Proposition \ref{pcpro2}]
The strategy of proof is exactly the same as in \cite[Theorem 6.5]{V}, and it takes advantage of a duality argument similar (although more involved) to the one employed in the proof of Proposition \ref{lemma-v-weak}. For the sake of completeness, we carry out the main computations, since this is a crucial result to our purposes. First of all we observe that, by subtracting the very weak formulations satisfied by $ \overline{u} $ and $ \underline{u} $, we obtain
\begin{equation}\label{eq-comp-diff-1}
	\int_0^{T} \int_E \left[ \left( \overline{u} - \underline{u} \right) \partial_t \xi + \left(\phi(\overline{u}) - \phi(\underline{u})  \right) \Delta \xi \right] d\V dt \le 0
\end{equation}
for all nonnegative $\xi \in C^2_c\!\left(\overline{E}\times [0, T) \right) $ with $ \xi=0 $ on $ \partial E \times (0,T) $. A standard cut-off argument in time then shows that \eqref{eq-comp-diff-1} implies
\begin{equation}\label{eq-comp-diff-2}
	\int_0^{\tau} \int_E \left[ \left( \overline{u} - \underline{u} \right) \partial_t \xi  + \left(\phi(\overline{u}) - \phi(\underline{u})  \right) \Delta \xi \right] d\V dt + \int_E \left[ \underline{u}(x,\tau)-\overline{u}(x,\tau) \right] \xi(x,\tau) \, d\V(x) \le 0 \, ,
\end{equation}
for almost every $ \tau  \in (0,T) $ and all  nonnegative $\xi \in C^2\!\left(\overline{E}\times [0, \tau] \right) $ with $ \xi=0 $ on $ \partial E \times (0,\tau) $. More precisely, one applies \eqref{eq-comp-diff-1} to the test functions $ \xi (x,t) \, \zeta_n(t) $, where is $\{ \zeta_n \}$ a smooth approximation of $ \chi_{[0,\tau]} (t)$, and lets $ n \to \infty $. Next, we introduce the function
\begin{equation*}\label{n4}
	a(x,t):=
	\begin{cases}
		\frac{\phi(\overline{u}(x,t)) - \phi(\underline{u}(x,t))}{\overline{u}(x,t) - \underline{u}(x,t)} & \text{if } \overline{u}(x,t) \neq \underline{u}(x,t) \, , \\
		0 & \text{if } \overline{u}(x,t) = \underline{u}(x,t) \, .
	\end{cases}
\end{equation*}
Note that, by construction, the nondecreasing monotonicity of $ \phi $ entails
\begin{equation}\label{a-finite}
	0 \le a(x,t) < +\infty \qquad \text{for a.e. } (x,t) \in E \times (0,T) \, ,
\end{equation}
although {a priori} $ a $ does not belong to any $L^p$ space. Hence, we can rewrite \eqref{eq-comp-diff-2} as follows:
\begin{equation}\label{eq-comp-diff-3}
	\int_0^{\tau} \int_E \left( \overline{u} - \underline{u} \right) \left[  \partial_t \xi   + a \, \Delta \xi \right] d\V dt + \int_E \left[ \underline{u}(x,\tau)-\overline{u}(x,\tau) \right] \xi(x,\tau) \, d\V(x) \le 0 \, .
\end{equation}
Given an arbitrary $ \varepsilon>0 $, we pick $ K_\varepsilon>\varepsilon $ so large that (let $\otimes$ denote the product of measures)
\begin{equation}\label{a-finite-bis}
	\left( \V \rfloor_{E} \otimes \mathcal{L}^1 \rfloor_{(0,T)} \right) \! \left( \left\{  a > K_\varepsilon \right\} \right) < \varepsilon^2 \, ,
\end{equation}
which is feasible by virtue of \eqref{a-finite}, where $ \mathcal{L}^1 $ stands for the one-dimensional Lebesgue measure. Next, we let
$$
a_\varepsilon := \left( a \vee \varepsilon \right) \wedge K_\varepsilon
$$
and pick a sequence of positive and bounded functions $ \{ a_n \} \subset C^\infty\!\left( \overline{E} \times [0,T] \right) $ such that
\begin{equation}\label{Linf-star}
	\lim_{n \to \infty} \left\| a_\varepsilon - a_n \right\|_{L^2\left( E \times (0,T) \right)} = 0 \, ,
\end{equation}
satisfying in addition the uniform bound from below
\begin{equation}\label{a-below}
	a_n \ge \frac{\varepsilon}{2} \, .
\end{equation}
Now, given an arbitrary nonnegative function $  \omega \in C^\infty_c(E)  $, for each $ n \in \mathbb{N} $ we consider the nonnegative solution $ \xi_n $ to the following dual backward parabolic problem:
\begin{equation}\label{n9}
	\begin{cases}
		\partial_t \xi_n + a_{n} \, \Delta \xi_n = 0 & \text{in } E \times (0, \tau) \, , \\
		\xi_n = 0 & \text{on } \partial E \times (0, \tau) \, , \\
		\xi_n = \omega & \text{on } E \times \{\tau\} \, .
	\end{cases}
\end{equation}
Due to the smoothness of the data and the fact that $ \operatorname{supp} \omega \Subset E $, standard parabolic regularity theory ensures that $ \xi_n $ is also a smooth function up to the boundary.
Hence, in view of \eqref{n9}, the admissible choice $ \xi=\xi_n $ in \eqref{eq-comp-diff-3} entails
\begin{equation}\label{eq-comp-diff-4}
	\int_0^{\tau} \int_E \left( \overline{u} - \underline{u} \right) \left( a - a_n \right) \Delta \xi_n  \, d\V dt + \int_E \left[ \underline{u}(x,\tau)-\overline{u}(x,\tau) \right] \omega(x) \, d\V(x) \le 0 \, .
\end{equation}
The crucial point is to suitably estimate the Laplacian of $\xi_n $, in order to handle the left term on the left-hand side of \eqref{eq-comp-diff-4}. To achieve it, we multiply the differential equation in \eqref{n9} by $\Delta\xi_{n}$ and integrate by parts in $ E \times (0, \tau)$, obtaining the identity
\begin{equation}\label{energy}
	\frac 1 2 \int_{E} \left| \nabla\xi_{n}(x,0) \right|^2  d\V(x) + \int_0^\tau \int_{E} a_{n} \left( \Delta \xi_{n} \right)^2 d\V dt = \frac 1 2 \int_{E} \left| \nabla \omega \right|^2 d\V \, .
\end{equation}
Hence, a routine application of the Cauchy-Schwarz inequality combined with \eqref{energy} yields
\begin{equation}\label{n48}
	\left|  \int_0^{\tau} \int_E \left( \overline{u} - \underline{u} \right) \left( a - a_n \right) \Delta \xi_n  \, d\V dt  \right|
	\le \frac{\left\| \nabla \omega  \right\|_{L^2(E)}}{\sqrt 2} \left[ \int_0^\tau \int_{E} \frac{\left( \overline{u} - \underline{u} \right)^2\left(a-a_{n}\right)^2}{a_{n}} \, {d}\V {d}t \right]^{\frac12} .
\end{equation}
Recalling the definition of $ a_\varepsilon $, along with  \eqref{a-finite-bis} and \eqref{a-below}, we can estimate the integral on the right-hand side of \eqref{n48} as follows:
$$
\begin{aligned}
	& \, \int_0^\tau \int_{E} \frac{\left( \overline{u} - \underline{u} \right)^2\left(a-a_{n}\right)^2}{a_{n}} \, {d}\V {d}t \\
	\le & \,  \frac{4}{\varepsilon} \int_0^\tau \int_{E} {\left( \overline{u} - \underline{u} \right)^2\left(a_\varepsilon-a_{n}\right)^2} \, {d}\V {d}t + \frac{4}{\varepsilon} \int_0^\tau \int_{E} {\left( \overline{u} - \underline{u} \right)^2\left(a-a_{\varepsilon}\right)^2} \, {d}\V {d}t \\
	\le & \, \frac{4}{\varepsilon} \left\| \overline{u} - \underline{u} \right\|_{L^\infty \left( E \times (0,T) \right)}^2  \left\| a_\varepsilon - a_n \right\|_{L^2\left( E \times (0,T) \right)}^2  + \frac{4}{\varepsilon} \iint_{ \left\{  a<\varepsilon \right\} } {\left( \overline{u} - \underline{u} \right)^2 \varepsilon^2 } \, {d}\V {d}t \\
	& + \frac{4}{\varepsilon} \iint_{ \left\{  a> K_\varepsilon \right\} } {\left( \overline{u} - \underline{u} \right)^2 a^2} \, {d}\V {d}t \\
	\le &  \, \frac{4}{\varepsilon} \left\| \overline{u} - \underline{u} \right\|_{L^\infty \left( E \times (0,T) \right)}^2  \left\| a_\varepsilon - a_n \right\|_{L^2\left( E \times (0,T) \right)}^2  + 4 \varepsilon \left\| \overline{u} - \underline{u} \right\|_{L^2 \left( E \times (0,T) \right)}^2 \\
	&  + 4 \varepsilon \left\| \phi(\overline{u}) - \phi(\underline{u}) \right\|_{L^\infty \left( E \times (0,T) \right)}^2  .
\end{aligned}
$$
Going back to \eqref{eq-comp-diff-4}, upon letting $n \to \infty $, using \eqref{Linf-star} and finally letting $ \varepsilon \to 0 $, we end up with
$$
\int_E \left[ \underline{u}(x,\tau)-\overline{u}(x,\tau) \right] \omega(x) \, d\V(x) \le 0 \, ,
$$
whence the thesis in view of the arbitrariness of $ \omega $ and $\tau$.
\end{proof}

\begin{proof}[Proof of Proposition \ref{exi-mini}]
	First of all, let
	\begin{equation}\label{eq52-k-defC}
	\mathsf{C} := \left\| u_0 \right\|_{L^\infty(M)} \vee \left\| \hat{u}_1 \right\|_{L^\infty(M \times (0,T))} \vee \left\| \hat{u}_2 \right\|_{L^\infty(M \times (0,T))} .
\end{equation}
Given an arbitrary regular exhaustion $ \left\{   E_k \right\} \subset \mathcal{P}(M)  $, we denote by $ u_k $ the corresponding solutions to the Cauchy-Dirichlet problems
\begin{equation}\label{eq52-k}
	\begin{cases}
		\partial_t u_k = \Delta  \phi(u_k)  & \text{in } E_k \times (0, +\infty) \, , \\
		\phi\!\left(u_k\right) =  \phi( - \mathsf{C} ) & \text{on } \partial E_k \times (0,+\infty) \, , \\
		u=u_0 & \text{on }  E_k \times \{0\} \,.
	\end{cases}
\end{equation}
Via the change of variables (and back) employed in Remark \ref{rem-nonhom}, with $ \alpha=-\mathsf{C} $, thanks to Proposition \ref{exiconstructed}(i) and (iii) we can assert that such solutions exist and comply with the uniform bounds
\begin{equation}\label{bound-k}
- \mathsf{C} \le	u_k  \leq \left\| u_0 + \mathsf{C} \right\|_{L^\infty(M)} - \mathsf{C}  \qquad \text{in } E_k \times (0, +\infty) \, ,
\end{equation}
for every $k \in \mathbb{N} $, where in the leftmost inequality of \eqref{bound-k} we used the fact that $ u_0 + \mathsf{C} \ge 0 $, thus $ u_k + \mathsf{C} \ge 0 $ owing to \eqref{cp-classic}. More specifically, they satisfy the very weak formulation of \eqref{eq52-k} according to Definition \ref{pap}, that is
\begin{equation}\label{eq53-minimal}
	\begin{aligned}
	\int_0^{+\infty} \int_{E_k} u_k \, \partial_t \xi \, d\V dt + \int_0^{+\infty}\int_{E_k} \phi(u_k) \, \Delta \xi \, d\V dt  =
	& \, - \int_{E_k} u_0(x) \, \xi(x,0)\, d\V(x) \\
	& \, +  \int_0^{+\infty}\int_{\partial E_k} \phi( - \mathsf{C} ) \, \frac{\partial \xi }{\partial \nu} \, d\s dt
	\end{aligned}
\end{equation}
for all $\xi \in C^2_c\!\left(\overline{E}_k \times [0, +\infty)\right)$ with $ \xi = 0 $ on $\partial E_k \times (0,+\infty)$. On the other hand, the energy estimate \eqref{energy-local-integ} ensures that
$$
\phi( u_k ) \in L^2 \! \left( (0,+\infty) ; H^1(E_k) \right) ,
$$
so that each $ \phi(u_k) $ has a well-defined trace on $ \partial E_k $ and more generally on $ \partial D  $ for every bounded smooth domain  $ D \Subset E_k $. As a result, it also satisfies
\begin{equation}\label{eq53-minimal-k}
	\begin{aligned}
		& \, \int_0^{+\infty} \int_{D} u_k \, \partial_t \xi \, d\V dt + \int_0^{+\infty}\int_{D} \phi(u_k) \, \Delta \xi \, d\V dt \\
		 = 		&  - \int_{D} u_0(x) \, \xi(x,0)\, d\V(x)  +  \int_0^{+\infty}\int_{\partial D} \phi( u_k ) \rfloor_{\partial D} \, \frac{\partial \xi }{\partial \nu} \, d\s dt \\
		 \le &  - \int_{D} u_0(x) \, \xi(x,0)\, d\V(x)  +  \int_0^{+\infty}\int_{\partial D} \phi( -\mathsf{C} )  \, \frac{\partial \xi }{\partial \nu} \, d\s dt
	\end{aligned}
\end{equation}
for every $D$ as above and all nonnegative $\xi \in C^2_c\!\left(D \times [0, +\infty)\right)$ with $ \xi = 0 $ on $\partial D \times (0,+\infty)$, where in the last inequality we took advantage of the lower bound in \eqref{bound-k}, along with the monotonicity of $\phi$ and the fact that $ \frac{\partial \xi }{\partial \nu} \le 0 $. Hence, by using \eqref{eq53-minimal-k} with $ D=E_{k-1} $, we infer that $ u_{k} $ is a very weak supersolution to problem \eqref{eq52-k} posed in $ E_{k-1} $ in the place of $ E_k $; therefore, a direct application of Proposition \ref{pcpro2} yields
\begin{equation}\label{bound-k+1}
	u_{k-1} \leq u_{k}  \qquad \text{in } E_{k-1} \times (0, +\infty) \, .
\end{equation}
The claimed solution $ u $ is simply obtained as $ u= \lim_{k \to \infty} u_k  $, where this limit is pointwise monotone in $k$ thanks to \eqref{bound-k+1}, up to extending each $ u_k $ to the value $ -\mathsf{C} $ outside $ E_k $, and uniformly bounded by \eqref{bound-k}. The fact that $u$ is indeed a very weak solution to \eqref{pp} (with $ T=+\infty $ actually), in the sense of Definition \ref{ds}, is readily seen by passing to the limit in the very weak formulations \eqref{eq53-minimal} tested against all fixed $\xi \in C^2_c( M \times [0,+\infty) )$, provided $ k $ is so large that $ \operatorname{supp} \xi \Subset E_k \times [0,+\infty) $.

We are left with proving that $u$ complies with \eqref{minimal-ineq}. Before, it is useful to notice that the limit function $u$ is independent of the exhaustion considered. To this aim, let $ \{ D_j \} \subset \mathcal{P}(M) $ be another regular exhaustion, and to avoid ambiguity let $ \left\{ \tilde u_j \right\} $ denote the analogues of the solutions $ \{ u_k \} $ to \eqref{eq52-k}, with $D_j$ replacing $ E_k $. We may then call $ \tilde{u} $ the corresponding limit solution obtained as above. Due to the definition of a regular exhaustion of $ M $, it is plain that we can find an increasing sequence\ $ \left\{ j_k \right\}$ such that for every $  k \in \mathbb{N} $ the precompact inclusion $E_k \Subset D_{j_k}$ holds. Therefore, the same argument that led to \eqref{bound-k+1} ensures that
$$
u_k \le \tilde{u}_{j_k} \qquad \text{in } E_k \times (0,+\infty) \, ,
$$
whence, by taking limits as $ k \to \infty $, we end up with
$$
u \le \tilde{u} \qquad \text{in } M \times (0,+\infty) \, .
$$
Upon reversing the roles of $ \{ E_k \} $ and $ \{ D_j \} $ we infer the opposite inequality as well, thus $ u=\tilde{u} $. Next we apply Proposition \ref{tams-prop} to the given solution $\hat{u}_{i}$ ($i=1,2$), which ensures the existence of a regular exhaustion $ \{ D_j \} \subset \mathcal{P}(M) $, possibly depending on $\hat{u}_i$, such that
\begin{equation}\label{es2002}
\begin{aligned}
 &  \int_0^{T}\int_{D_j} \hat{u}_i \, \partial_t \xi \, d\V dt +  \int_0^{T} \int_{D_j} \phi\!\left( \hat{u}_i \right) \Delta \xi \, d\V dt \\
 = &  - \int_{D_j} u_{0}(x,0) \, \xi(x,0) \, d\V(x) + \int_0^{T} \int_{\partial D_j} \phi\!\left( \hat{u}_i \right) \rfloor_{D_j} \frac{\partial \xi}{\partial \nu} \, d\s dt  \\
\le  &  - \int_{D_j} u_{0}(x,0) \, \xi(x,0) \, d\V(x) + \int_0^{T} \int_{\partial D_j} \phi( -\mathsf{C} ) \, \frac{\partial \xi}{\partial \nu} \, d\s dt \, ,
 \end{aligned}
\end{equation}
for every $ j \in \mathbb{N} $ and all nonnegative test function $ \xi \in C^2_c\!\left( \overline{D}_j \times [0,T) \right)  $ with $ \xi = 0 $ on $\partial D_j \times (0,T)$, where in the last passage we used \eqref{eq52-k-defC}, the monotonicity of $ \phi $ and again the fact that outer normal derivative of $ \xi $ on $ \partial D_j $ is nonpositive. Similarly to \eqref{eq53-minimal-k}, inequality \eqref{es2002} establishes that $ \hat{u}_i \rfloor_{D_j \times (0,T)} $ is a very weak supersolution, in the sense of Definition \ref{pap}, to the Cauchy-Dirichlet problem satisfied by $ \tilde{u}_j $ (keep in mind the above notation). Hence, Proposition \ref{pcpro2} is again applicable, ensuring that
	\begin{equation}\label{u-j-lim}
	\tilde{u}_j	\le \hat{u}_i \qquad \text{in } D_j \times (0,T)
	\end{equation}
for every $ j \in \mathbb{N} $. The thesis then follows by taking limits as $j \to \infty $ in \eqref{u-j-lim}, recalling that $ \{ \tilde{u}_j \} $ converges to $u$.
\end{proof}

\begin{proof}[Proof of Proposition \ref{exi-mini-real}]
	The argument follows closely the ideas behind the proof of Proposition \ref{exi-mini}. Indeed, it is enough to construct the claimed minimal solution as the pointwise limit of the solutions $ u_k $ to the homogeneous Cauchy-Dirichlet problems (let $ E_k \subset \mathcal{P}(M) $ be a given regular exhaustion)
	\begin{equation*}\label{eq52-k-hom}
		\begin{cases}
			\partial_t u_k = \Delta  \phi(u_k)  & \text{in } E_k \times (0, +\infty) \, , \\
			\phi\!\left(u_k\right) =  0 & \text{on } \partial E_k \times (0,+\infty) \, , \\
			u=u_0 \ge 0 & \text{on }  E_k \times \{0\} \, ,
		\end{cases}
	\end{equation*}
	assuming with no loss of generality that $ \phi(0)=0 $. From Proposition \ref{exiconstructed}(i) and (iii), we can readily infer that
	\begin{equation*}\label{bound-k-positive}
		0  \le	u_k  \leq \left\| u_0  \right\|_{L^\infty(M)}  \qquad \text{in } E_k \times (0, +\infty) \, ,
	\end{equation*}
for every $ k \in \mathbb{N} $. Moreover, the monotonicity inequalities \eqref{bound-k+1} still hold. As a result, the nonnegative and bounded function $ u = \lim_{k \to \infty} u_k $ is a very weak solution to \eqref{pp}. The fact that $u$ is smaller than any other nonnegative very weak solution $v$ to the same Cauchy problem just follows by showing that $ u $ is independent of the chosen exhaustion, and repeating the final part of the proof of Proposition \ref{exi-mini} with $v$ replacing $\hat{u}_i$. It is plain that $T$ plays no role in the construction of $u$. In the case $ \phi(0) \neq 0 $, it is enough to apply the above argument to $ \phi_0 := \phi - \phi(0) $, upon observing that $ u,v $ are very weak solutions to \eqref{pp} if and only if they are very weak solutions to the same problem with $ \phi $ replaced by $ \phi + c $, for any constant $ c \in \mathbb{R} $.
\end{proof}

\section{Proofs of the main results}\label{proof}

In this section, we will first establish Theorem \ref{teo-ell} (along with Theorem \ref{cor-gen-prs}) and then Theorem \ref{teo-par}, as the latter, at least in one implication, takes advantage of the former.

\begin{proof}[Proof of Theorem \ref{teo-ell}]
\ \smallskip \\
\noindent \eqref{j}$ \Rightarrow $\eqref{jj}  Since $\psi$ is, in particular, locally bounded, for every $ A>0 $ we can find $K>0$ so large that
\begin{equation}\label{eq1}
\psi(w) \leq  K A  \qquad \forall  w \in [0,2A] \, .
\end{equation}
In view of hypothesis \eqref{j}, and using \cite[Theorem 8.18]{GrigHK}, we can assert that there exists a nonnegative and bounded function $U$ such that
\begin{equation*}
\Delta U =  K U \quad \text{in } M \qquad \text{and} \qquad \left\| U \right\|_{L^\infty(M)} =  2A \, .
\end{equation*}
Let us consider the function
$$
 \underline{W} := U - A \, .
$$
Due to \eqref{eq1}, the fact that $ 0 \le U \le 2A$ and the monotonicity of $\psi$, we have:
\begin{equation}\label{eq6-ter}
\Delta \underline{W} = \Delta U = KU \ge  K A \ge \psi(U) \ge \psi\!\left(U-A \right) = \psi(\underline{W}) \qquad \text{where } U>A
\end{equation}
and
\begin{equation}\label{eq6-quater}
\Delta \underline{W} = \Delta U = KU \ge 0 \ge \psi\!\left(U-A \right) = \psi(\underline{W}) \qquad \text{where } U \le A \, ,
\end{equation}
recalling that $ \psi(0)=0 $ and thus $ \psi(w) \le 0 $ for $ w \le 0 $. As a result, we can deduce that the function $ \underline{W} $ satisfies
$$
\Delta \underline{W}  \ge \psi(\underline{W}) \qquad \text{in } M \, ,
$$
namely it is a very weak subsolution to \eqref{eqe} (in fact it is more regular than that). Moreover, by construction,
\begin{equation*}\label{eq6-AA}
\left\| \underline{W}^+ \right\|_{L^\infty(M)} = A \, .
\end{equation*}
Hence, we are in position to invoke Proposition \ref{from-sub-to-sol}, which guarantees that a solution $ W $ to \eqref{eqe} exists and fulfills \eqref{new-bound}. In particular, $W$ is nonnegative and complies with the identity $ \left\| W \right\|_{\infty} = A $. Because $ A>0 $ is a free parameter, different values of $A$ yield different solutions, thus \eqref{jj} follows.

\smallskip
\noindent \eqref{jj}$ \Rightarrow $\eqref{jjj}  It is obvious.

\smallskip
\noindent \eqref{jjj}$ \Rightarrow $\eqref{j} Hypothesis \eqref{jjj} ensures the existence of a nonnegative, nontrivial, bounded function $W$ satisfying
\begin{equation*}\label{eq7}
\Delta W \ge \psi(W) \qquad \text{in }  M \, .
\end{equation*}
Let us set
\[
\mathsf{c}:= \left\| W \right\|_{L^\infty(M)} > 0 \,  ,
\]
and define
\[
\underline{U} := W-\tfrac{\mathsf{c}}{2} \, .
\]
Similarly to \eqref{eq6-ter} and \eqref{eq6-quater}, we have:
\begin{equation}\label{eqA6-1}
\Delta \underline{U} = \Delta W \ge \psi(W) \ge \psi\!\left( \tfrac{\mathsf c}{2} \right) = \frac{\psi\!\left( \tfrac{\mathsf c}{2} \right) }{\tfrac {\mathsf c} 2} \cdot \tfrac{\mathsf c}{2} \ge \frac{\psi\!\left( \tfrac{\mathsf c}{2} \right) }{\frac {\mathsf c} 2} \,  \underline{U} \qquad \text{where } W>\tfrac{ \mathsf c } 2
\end{equation}
and
\begin{equation}\label{eqA6-2}
\Delta \underline{U} = \Delta W \ge \psi(W) \ge \psi\!\left( 0 \right) = 0 \ge \frac{\psi\!\left( \tfrac{\mathsf c}{2} \right) }{\frac {\mathsf c} 2} \,  \underline{U} \qquad \text{where } W \le \tfrac{ \mathsf c } 2 \, .
\end{equation}
Note that in \eqref{eqA6-1} and \eqref{eqA6-2} we have exploited the monotonicity of $ \psi $ along with the fact that $ \underline{U} \le \frac{\mathsf c}{2} $ by construction. We can therefore deduce that the function  $ \underline{U} $, which is nontrivial and bounded, satisfies
\begin{equation*}%\label{eqA6-3}
\Delta \underline{U} \ge \lambda \, \underline{U} \qquad \text{in } M
\end{equation*}
along with $ \left\| \underline{U}^+ \right\| _\infty = \frac{\mathsf{c}} 2$, where
$$
\lambda = \frac{\psi\!\left(\frac {\mathsf c}2\right)}{\frac {\mathsf c}2}>0 \, ,
$$
the positivity of the parameter $ \lambda $ being ensured by the fact that $ \psi \in \mathfrak{C}^+ $. Hence, upon applying again Proposition \ref{from-sub-to-sol} in the special case $ \psi(w)=\lambda w $, we obtain a nonnegative, nontrivial, bounded solution $ U $ to $  \Delta U  = \lambda \, U $, which means that $ M $ is stochastically incomplete.
\end{proof}

\begin{proof}[Proof of Theorem \ref{cor-gen-prs}]
\
\smallskip
\\
\noindent \eqref{k}$ \Rightarrow $\eqref{kk} In view of the assumptions on $ f  $, it is plain that the auxiliary function
$$
\tilde{f}(w) := \max_{[w_0,w]} f \qquad \forall w \in [w_0,w_1]
$$
is continuous, nondecreasing and satisfies
\begin{equation}\label{f-g}
\tilde{f}(w_0)=0  \,  , \qquad  \tilde{f}>0 \quad \text{in } (w_0,w_1] \, , \qquad \tilde{f} \ge f \quad \text{in } [w_0,w_1] \, .
\end{equation}
Therefore, the function
\begin{equation}\label{f-g-psi}
\psi(z) :=
\begin{cases}
	0 & \text{if } z < 0 \, , \\
\tilde{f}(z+w_0)   & \text{if } z \in [0,w_1-w_0] \, ,    \\
 \tilde{f}(w_1)  &  \text{if } z>w_1-w_0 \, ,
 \end{cases}
\end{equation}
falls within the class $ \mathfrak{C}^+ $, so that Theorem  \ref{teo-ell}\eqref{j}$ \Rightarrow $\eqref{jj} is applicable, ensuring the existence of a nonnegative, nontrivial and bounded function $ W $ such that
\begin{equation}\label{eq-W}
\Delta W = \psi(W) \qquad \text{in } M \, .
\end{equation}
Moreover, from the corresponding proof, one can see that the $ L^\infty(M) $ norm of $W$ can be tuned to be arbitrarily small; in particular, we may assume that $ \left\| W \right\|_{\infty} \le w_1-w_0  $. Let
$$
w := W+w_0  \, .
$$
By construction, the (essential) image of $ w $ lies in the interval $ [w_0,w_1] $; hence, due to \eqref{f-g}, \eqref{f-g-psi}, \eqref{eq-W} and the definition of $  \psi $, we end up with
$$
\Delta w = \Delta W = \psi(W) = \tilde{f}(W+w_0) = \tilde{f}(w) \ge f(w) \qquad \text{in } M \, .
$$
Also, we observe that $ f(w^\ast)>0 $ because $ f $ is strictly positive in $ (w_0,w_1] $. We have therefore constructed a function $w$ which complies with all the requirements of \eqref{kk} (it has even stronger properties).

\smallskip
\noindent \eqref{kk}$ \Rightarrow $\eqref{k} Let $ w \in L^1_{\mathrm{loc}}(M) $ be as in the statement. First of all, we define the following auxiliary function:
\begin{equation*}\label{f-g-psi-bis}
	\hat{f}(w) :=
	\begin{cases}
		w-w^* + \min_{[w,w^*]} f   & \text{if } w \le  w^*  \, ,   \\
		{f}\!\left(w^*\right)  &  \text{if } w>w^* \, .
	\end{cases}
\end{equation*}
It is plain that $ \hat{f} $ is continuous, nondecreasing and satisfies
\begin{equation*}\label{f-g-bis}
 \hat{f}\!\left(w^*\right) =  {f}\!\left(w^*\right) > 0 \qquad \text{and} \qquad \hat{f} \le f \quad \text{in } \left(-\infty,w^*\right]  .
\end{equation*}
Moreover, by construction, it admits a unique zero, that we call $w_0$. Clearly, $ w_0 < w^\ast $. Because $ \hat{f} $ is less than $f$ in the essential image of $w$, we have that $w$ also satisfies
\begin{equation*}\label{f-w-bis-bis}
	\Delta w \ge \hat f(w)  \qquad \text{in $M$}
\end{equation*}
in the dense of distributions or, equivalently, is a very weak subsolution in the sense of Definition \ref{vw}. Hence, we deduce that the function $ \underline{W} := w-w_0 $ is in turn a very weak subsolution to
\begin{equation*}\label{f-w-ter}
	\Delta \underline W \ge \hat f(\underline W + w_0) =: \psi\!\left( \underline{W} \right)  \qquad \text{in $M$} \, .
\end{equation*}
Since $ \psi \in \mathfrak{C}^+ $, we are in position to apply Proposition \ref{from-sub-to-sol}, which yields the existence of a very weak solution $W$ to \eqref{eqe} that complies with
\begin{equation*}\label{new-bound-soll}
\left( w-w_0 \right)^+	= \underline{W}^+ \le W \le \left\|  \underline{W}^+ \right\|_{L^\infty(M)} = w^\ast - w_0 \qquad \text{in } M \, .
\end{equation*}
The function $W$ being nonnegative and nontrivial, we can conclude that $ M $ is stochastically incomplete as a consequence of Theorem \ref{teo-ell}\eqref{jjj}$ \Rightarrow $\eqref{j}.

\smallskip

\noindent (d$^*$)$ \Rightarrow $\eqref{kk} It is obvious.

\smallskip

\noindent \eqref{k}$ \Rightarrow $(d$^*$) It is enough to observe that, if $ f $ is in addition nondecreasing in the interval $ [w_0,w_1] $, in the proof of \eqref{k}$ \Rightarrow$\eqref{kk} the function $ \tilde{f} $ coincides with $ f $, hence we end up with a solution $w$ to
$$
\Delta w = f(w) \qquad \text{in } M  \, ,
$$
which by construction satisfies $ w_0 \le w \le w_1 $ along with \eqref{f-w-bis}. Finally, the fact that one can actually construct infinitely many such solutions still follows from the arbitrariness of the $ L^\infty(M) $ norm of the auxiliary solution $W$ used in \eqref{k}$ \Rightarrow$\eqref{kk}.
\end{proof}

Before providing the proof of Theorem \ref{teo-par}, we need to recall a useful property in real analysis concerning the concavity of the modulus of continuity, that we prove for the sake of completeness.

\begin{lem}\label{concave-mc}
Let $ \phi: I \to \mathbb{R} $ be a continuous function, where $I$ is a closed and bounded interval. Then there exists a concave, continuous and strictly increasing function $ F: [0,+\infty) \to [0,+\infty) $ such that $ F(0)=0 $, $ \lim_{z \to +\infty} F(z)=+\infty $ and
\begin{equation}\label{cmc}
\left| \phi(u) - \phi(v) \right| \le F\!\left( \left| u-v \right| \right) \qquad \forall u,v \in I \, .
\end{equation}
\end{lem}
\begin{proof}
Since $ \phi $ is continuous in a compact interval, it is well known that its modulus of continuity
$$
\omega (\delta) := \max_{u,v \in I: \, |u-v| \le \delta} \left| \phi(u) - \phi(v) \right|  \qquad \forall \delta \ge 0
$$
is in turn a continuous function from $ [0, + \infty) $ to $ [0,+\infty) $ with $\omega(0)=0$, which is also nondecreasing and satisfies
\begin{equation*}%\label{cmc-1}
\left| \phi(u) - \phi(v) \right| \le \omega\!\left( \left| u-v \right| \right) \qquad \forall u,v \in I \, .
\end{equation*}
In order to exhibit a function $F$ as in the statement, one can simply take the concave envelope of (the graph of) $ \omega $ and add to it the identity function:
$$
F(z) := z + \sup_{\lambda \in [0,1] , \, z_1 , z_2 \in [0,+\infty)} \left\{ \lambda \, \omega(z_1) + (1-\lambda) \, \omega(z_2) : \ z = \lambda \, z_1 + (1-\lambda) \, z_2 \right\} \qquad \forall z \ge 0 \, .
$$
It is not difficult to check that $ F $ is still continuous, strictly increasing and fulfills $ F(0)=0 $ and $ \lim_{z \to +\infty} F(z)=+\infty $. Moreover, it satisfies \eqref{cmc} (just note that $ \omega \le F $) and it is concave by construction.
\end{proof}

\begin{proof}[Proof of Theorem \ref{teo-par}]
\ \smallskip \\
\noindent \eqref{i}$ \Rightarrow $\eqref{ii} Since $ M $ is stochastically  incomplete, we know that there exists a nonnegative, nontrivial and bounded function $ U $ such that
\begin{equation}\label{basic-u}
\Delta U= U \qquad \text{in } M \, .
\end{equation}
For future convenience, and without loss of generality, we will assume that
\begin{equation*}\label{E1}
\left\| U \right\|_{L^\infty(M)} = 1 \, ,
\end{equation*}
which is always feasible by the linearity of \eqref{basic-u}. Let us set
\[
\Omega:=\left\{x\in M : \  U(x)>\tfrac 12 \right\} .
\]
 Because $U$ is at least continuous (it is in fact smooth), we observe that $ \Omega $ is open in $ M $, although we will never exploit this property explicitly. What actually matters is that it is a set of positive measure, along with all the superlevel sets of $ U $ up to the level $1$. Let us define
\[
V := 2 \left( U -\tfrac 1 2 \right) .
\]
By construction, the function $ V $ is also smooth and fulfills
\begin{equation}\label{eq-propV}
0 <V \le 1 \quad \text{in } \Omega \, , \qquad  V \le 0 \quad \text{in } \Omega^c \, , \qquad \left\| V \right\|_{L^\infty(\Omega)} = 1 \,  .
\end{equation}
Since $ \Delta U \ge \frac{1}{2} $ in $ \Omega $, we infer that $ V $ satisfies
\begin{equation}\label{system-V}
\Delta V \ge 1  \qquad \text{in }  \Omega \, ;
\end{equation}
in particular, it admits no maximum in $\Omega$ (hence $V \le 1$ is in fact a strict inequality). We now introduce, by means of $V$, two auxiliary functions that will be key to our purposes. Given an arbitrary $ \alpha \in \mathbb{R} $, we let $ \beta = \phi(\alpha)  $, and for a constant $c>0$ to be chosen later we set
\begin{equation}\label{def-z}
	Z:= c \left( V -  1 \right) +  \beta \, .
\end{equation}
Thanks to \eqref{eq-propV}, we have that
\begin{equation}\label{eq27-bis}
Z < \beta \quad \text{in } {\Omega} \, , \qquad  Z \le  -c+ \beta \quad \text{in } \Omega^c \, ,
\end{equation}
whereas from \eqref{system-V} and \eqref{basic-u} it follows that
\begin{equation}\label{eq27-quater}
	\Delta Z \ge c \quad \text{in } \Omega \, , \qquad \Delta Z \ge 0 \quad \text{in } \Omega^c  \, .
\end{equation}
Similarly, for another constant $ \kappa>0 $ to be chosen later, we set
\begin{equation}\label{def-y}
	Y := \kappa \left( 1 -  V \right) +  \beta \, .
\end{equation}
As in \eqref{eq27-bis} and \eqref{eq27-quater}, it is readily seen that $Y$ satisfies
\begin{equation*}\label{eq27-bis.Y}
	Y > \beta \quad \text{in } {\Omega} \, , \qquad  Y \ge  \kappa + \beta \quad \text{in } \Omega^c \, ,
\end{equation*}
and
\begin{equation*}\label{eq27-quater-Y}
	-\Delta Y \ge \kappa \quad \text{in } \Omega \, , \qquad -\Delta Y \ge 0 \quad \text{in } \Omega^c  \, .
\end{equation*}
Let $ \{ E_k \} \subset \mathcal{P}(M) $ be a regular exhaustion. For every $i,j \in \mathbb{N}$, we consider the following approximate initial data:
\begin{equation}\label{approx-datum}
u_{0,i,j} := \alpha + (u_0-\alpha)^+ \chi_{E_i} - (u_0-\alpha)^-\chi_{E_j} \, ;
\end{equation}
it is readily seen that
\begin{equation}\label{approx-datum-monotone}
u_0 \wedge \alpha \le u_{0,i,j} \le u_{0,i+1,j}  \, , \qquad 	u_{0,i,j+1} \le u_{0,i,j} \le u_0 \vee \alpha \, .
\end{equation}
Consider the following Cauchy-Dirichlet problems:
\begin{equation}\label{cd-up}
\begin{cases}
\partial_t u_{k,i,j} = \Delta \phi\!\left(u_{k,i,j}\right) & \text{in } E_k \times (0,+\infty) \, ,\\
\phi(u_{k,i,j}) = \beta  & \text{on } \partial E_k \times (0, +\infty ) \, ,\\
u_{k,i,j}=u_{0,i,j} & \text{on } E_k \times \{0\} \, .
\end{cases}
\end{equation}
As in Remark \ref{rem-nonhom}, upon setting $ \phi_\alpha (v) := \phi(v+\alpha) - \beta  $, $ v_{0,i,j} := u_{0,i,j} - \alpha $, $ v_{k,i,j} := u_{k,i,j} - \alpha$, it is plain that $ u_{k,i,j}  $ is the solution to \eqref{cd-up} if and only if $ v_{i,j,k} $ is the solution to the homogeneous problem
\begin{equation}\label{cd-up-v}
	\begin{cases}
		\partial_t v_{k,i,j} = \Delta \phi_\alpha\!\left(v_{k,i,j}\right) & \text{in } E_k \times (0,+\infty) \, ,\\
	    \phi_\alpha(v_{k,i,j}) = 0  & \text{on } \partial E_k \times (0, +\infty ) \, ,\\
		v_{k,i,j}=v_{0,i,j} & \text{on } E_k \times \{0\} \, .
	\end{cases}
\end{equation}
Proposition \ref{exiconstructed} is thus applicable to \eqref{cd-up-v}, guaranteeing that the corresponding very weak solutions exist (in the sense of Definition \ref{pap}); moreover, in view of the monotonicity inequalities in  \eqref{approx-datum-monotone} and item (iii) of the just cited proposition, such solutions preserve the ordering of the initial data:
\begin{equation}\label{ord-v-k}
v_{k,i,j} \le v_{k,i+1,j}  \quad \text{and} \quad v_{k,i,j+1} \le v_{k,i,j} \qquad \text{in } E_k \times (0,+\infty) \, .
\end{equation}
In addition, since
$$
- \left\| u_0-\alpha \right\|_{L^\infty(M)} \le -\left( u_0 - \alpha \right)^- \le v_{0,i,j} \le \left( u_0 - \alpha \right)^+   \le \left\| u_0-\alpha \right\|_{L^\infty(M)} ,
$$
as a straightforward consequence of Proposition \ref{pcpro2} we infer that
\begin{equation}\label{ord-v-k-bdd}
- \left\| u_0-\alpha \right\|_{L^\infty(M)} \le v_{k,i,j} \le  \left\| u_0-\alpha \right\|_{L^\infty(M)} \qquad \text{in } E_k \times (0,+\infty) \, ,
\end{equation}
because the constant functions $ - \left\| u_0-\alpha \right\|_{\infty} $ and $\left\| u_0-\alpha \right\|_{\infty}$ are a subsolution and, respectively, a supersolution to \eqref{cd-up-v}. Note that, by construction, we have that each $ v_{0,i,j} $ is bounded with compact support in $M$. Hence, we can appeal to Proposition \ref{exi-sign-ch}, which ensures that there exists some function $v_{i,j}$ such that (up to subsequences and the obvious extensions outside $E_k$)
\begin{equation}\label{prop-bre-approx-1}
	v_{k,i,j} \underset{k \to \infty}{\longrightarrow} v_{i,j} \ \ \text{weakly$^*$ in } L^\infty(M \times (0,+\infty)) \ \ \ \text{and} \ \ \ \phi_\alpha\!\left(v_{k,i,j}\right) \underset{k \to \infty}{\longrightarrow} \phi_\alpha\!\left(v_{i,j}\right) \ \ \text{in } L^2_{\mathrm{loc}}(M \times [0,+\infty))  ;
\end{equation}
furthermore, each $v_{i,j}$ is a very weak solution  to the Cauchy problem
\begin{equation}\label{pp-1}
	\begin{cases}
	\partial_t	v_{i,j} =\Delta \phi_\alpha\!\left(v_{i,j}\right) & \text{in } M \times(0,+\infty) \, , \\
		v_{i,j} = v_{0,i,j} & \text{on } M \times \{ 0 \} \, .
	\end{cases}
\end{equation}
Note that, via a standard diagonal procedure, the subsequence under which \eqref{prop-bre-approx-1} holds can be taken to be independent of $i$ and $j$. Therefore, as a consequence of \eqref{ord-v-k}, the solutions $ \left\{  v_{i,j} \right\} $ preserve ordering:
\begin{equation*}\label{ord-v-ij}
	v_{i,j} \le v_{i+1,j}  \quad \text{and} \quad v_{i,j+1} \le v_{i,j} \qquad \text{in } M \times (0,+\infty) \, ,
\end{equation*}
as well as uniform boundedness:
\begin{equation*}\label{ord-v-ij-bdd}
	- \left\| u_0-\alpha \right\|_{L^\infty(M)} \le v_{i,j} \le  \left\| u_0-\alpha \right\|_{L^\infty(M)} \qquad \text{in } M \times (0,+\infty) \, .
\end{equation*}
Hence, by letting first $ i \to \infty $ and then $ j \to \infty $, recalling \eqref{approx-datum}, we can pass to the limit twice in \eqref{pp-1} by monotonicity (first increasing then decreasing), to obtain a very weak solution $v$ to the Cauchy problem
\begin{equation*}\label{pp-2}
	\begin{cases}
		\partial_t	v =\Delta \phi_\alpha\!\left(v\right) & \text{in } M \times(0,+\infty) \, , \\
		v = u_0-\alpha & \text{on } M \times \{ 0 \} \, .
	\end{cases}
\end{equation*}
By undoing the changes of variables $ \phi(u) = \phi_\alpha (u-\alpha) + \beta  $, $ u = v + \alpha$ and integrating by parts, we easily infer that $u$ is a very weak solution to the original Cauchy problem \eqref{pp} with $ T=+\infty $. Our next goal is to show that, by virtue of the boundary condition prescribed in \eqref{cd-up}, such a solution has a specific spatial asymptotic behavior. More precisely, we claim that
\begin{equation}\label{eee2-limit}
	\int_0^S \phi(u(\cdot,t)) \, dt  \ge S Z \qquad \text{in } M
\end{equation}
and
\begin{equation}\label{eee2-limit-below}
	\int_0^S \phi(u(\cdot,t)) \, dt  \le S Y \qquad \text{in } M
\end{equation}
for suitable choices of the parameters $ c $ and $\kappa$, where $ S \in (0,T) $ is fixed once for all and we recall that $ Z$ (resp.~$Y $) is defined in \eqref{def-z} (resp.~\eqref{def-y}). First, let us explain why the claim yields the thesis. For every $ \varepsilon>0 $, let
\begin{equation}\label{def-omega-eps}
\Omega_\varepsilon := \left\{ x \in M : \  1-\varepsilon <  V(x) \le 1  \right\} ,
\end{equation}
which has positive measure by construction. Due to \eqref{eee2-limit} and \eqref{eee2-limit-below}, we can deduce that
\begin{equation}\label{omega-eps}
 S \left( - c \, \varepsilon + \beta \right)  \le	\int_0^S \phi(u(\cdot,t)) \, dt  \le S \left( \kappa \, \varepsilon + \beta \right) \qquad \text{in } \Omega_\varepsilon \, .
\end{equation}
From \eqref{omega-eps} and the arbitrariness of $\varepsilon$ (note that $ \Omega_\varepsilon $ is independent of $ \beta $), it is apparent that different values of $ \beta $ give rise to different solutions; because $\beta$ is an arbitrary element of $ \phi(\mathbb{R}) $ and $ \phi $ is a continuous and nonconstant function, the thesis follows.

We are thus left with establishing \eqref{eee2-limit} and \eqref{eee2-limit-below}; to this end, we introduce the integral functions
$$
\mathcal  F_{k,i,j} := \int_0^S \phi\!\left(u_{k,i,j}(\cdot,t)\right) dt \, .
$$
Let us start from \eqref{eee2-limit}. Upon integrating \eqref{cd-up} in time, one sees that these functions satisfy the elliptic problems
\begin{equation}\label{eq31-modif}
\begin{cases}
\Delta \mathcal F_{k,i,j} = u_{k,i,j}(\cdot, S)- u_{0,i,j}  & \text{in } E_k \, ,\\
\mathcal F_{k,i,j} = S \beta   & \text{on } \partial E_k \, .
\end{cases}
\end{equation}
We point out that $ u_{k,i,j}(\cdot,S) $ is well defined thanks to the continuity properties of the solutions $ v_{k,i,j} $, which follow from the statement of Proposition \ref{exiconstructed}. Moreover, by virtue of \eqref{energy-local-integ} (applied to $ v_{k,i,j} $), we can assert that $ \mathcal F_{k,i,j} \in H^1(E_k) $, thus \eqref{eq31-modif} is certainly satisfied at least in the weak sense. Due to \eqref{ord-v-k-bdd}, we have that
\begin{equation}\label{eq31-ineq-u}
\Delta \mathcal F_{k,i,j} \le 2 \left\| u_0 - \alpha \right\|_{L^\infty(M)}  \quad \text{and} \quad \mathcal F_{k,i,j} \ge S \, \phi\big( \alpha -\left\| u_0 - \alpha \right\|_{L^\infty(M)} \big) \qquad \text{in } E_k \, .
\end{equation}
Now, if we choose $ c>0 $ to be any constant satisfying
\begin{equation}\label{eq31-ineq-z-in}
 S c \ge  2 \left\| u_0 - \alpha \right\|_{L^\infty(M)} \qquad  \text{and} \qquad - c+ \beta  \le \phi\big( \alpha -\left\| u_0 - \alpha \right\|_{L^\infty(M)} \big) \, ,
\end{equation}
and let
$$
G := S Z - \mathcal F_{k,i,j} \in H^1(E_k) \, ,
$$
thanks to \eqref{eq27-quater} and the left inequalities in \eqref{eq31-ineq-u} and \eqref{eq31-ineq-z-in} we see that the latter function satisfies
\begin{equation}\label{eee1}
\Delta  G \ge H :=
\begin{cases}
0 & \text{in } \Omega \cap E_k \, , \\
- 2 \left\| u_0 - \alpha \right\|_{L^\infty(M)}  & \text{in } E_k \setminus \Omega \, ,
\end{cases}
\end{equation}
weakly in $ E_k $. Furthermore, thanks to \eqref{eq27-bis}, the boundary condition in \eqref{eq31-modif}, the right inequalities in \eqref{eq31-ineq-u} and \eqref{eq31-ineq-z-in}, it is plain that $ G \le 0 $ in $ \left( E_k \setminus \Omega \right) \cup \partial E_k $. Therefore, if we test \eqref{eee1} with the function $ -G^+ \in H^1_0(E_k) $, we end up with
\begin{equation}\label{subsol-A}
\int_{E_k} \left| \nabla  G^+ \right|^2 \, d\V \le  - \int_{E_k} H \, G^+ \, d\V = - \int_{\Omega \cap E_k} H \, G^+ \, d\V = 0 \, ,
\end{equation}
which readily entails $ G^+ = 0 $, that is $ SZ \le \mathcal F_{k,i,j} $ in the whole $ E_k $:
\begin{equation}\label{eee2}
 \int_0^S \phi\!\left(u_{k,i,j}(\cdot,t)\right)  dt  \ge S Z \qquad \text{in }  E_k \, .
\end{equation}
In view of \eqref{prop-bre-approx-1} and the monotonicity of $ \left\{ u_{k,i,j} \right\} $ with respect to the indexes $i$ and $j$, it is plain that \eqref{eee2} is stable under passage to the limit as $ k,i,j \to \infty $, whence \eqref{eee2-limit}. Finally, let us deal with \eqref{eee2-limit-below}. Still in view of \eqref{ord-v-k-bdd}, we have that
\begin{equation*}\label{eq31-ineq-u-below}
	\Delta \mathcal F_{k,i,j} \ge -2 \left\| u_0 - \alpha \right\|_{L^\infty(M)}  \quad \text{and} \quad \mathcal F_{k,i,j} \le S \, \phi\big( \alpha +\left\| u_0 - \alpha \right\|_{L^\infty(M)} \big) \qquad \text{in } E_k \, .
\end{equation*}
Therefore, if we let $\kappa$ be so large that
$$
 S \kappa \ge  2 \left\| u_0 - \alpha \right\|_{L^\infty(M)} \qquad  \text{and} \qquad \kappa + \beta  \ge \phi\big( \alpha  + \left\| u_0 - \alpha \right\|_{L^\infty(M)} \big) \, ,
$$
set
$$
J := \mathcal F_{k,i,j} - S Y  \in H^1(E_k)
$$
and reason exactly as above, we find that also $J$ satisfies
\begin{equation*}\label{eee1-below}
	\Delta  J \ge H \qquad \text{weaky in } E_k \, .
\end{equation*}
Since $ J \le 0 $ in $ \left( E_k \setminus \Omega \right) \cup \partial E_k $, by applying \eqref{subsol-A} with $G$ replaced by $J$ we infer that $ J^+=0 $, namely  $ \mathcal F_{k,i,j} \le SY  $ in $ E_k $, i.e.~the analogue of \eqref{eee2} from above. Upon taking limits as $ k,i,j \to \infty $, we reach \eqref{eee2-limit-below}.

\smallskip
\noindent \eqref{ii}$ \Rightarrow $\eqref{iii} It is obvious.

\smallskip
\noindent \eqref{iii}$ \Rightarrow $\eqref{i} Thanks to Proposition \ref{exi-mini}, we can assume without loss of generality that there exist two different solutions $ u,v \in L^\infty(M \times (0,T)) $ to the Cauchy problem \eqref{pp} which are also ordered, say $ v \ge u $. By multiplying $ (v-u) $ by $ e^{-t} $ and differentiating with respect to time, we obtain the following identity:
\begin{equation}\label{eq:ww1}
\partial_t \! \left[ e^{-t} \left( v - u \right) \right] = e^{-t} \left[ \Delta \phi(v)  - \Delta \phi(u) \right] - e^{-t} \left( v - u \right) ,
\end{equation}
to be interpreted in the very weak sense over $ M \times (0,T) $. Now, let us integrate \eqref{eq:ww1} between $ t=0 $ and $ t = T $. By exploiting the fact that $ v \ge u $ and $ v(\cdot,0) = u(\cdot,0) = u_0 $, we end up with the inequality
\begin{equation}\label{eq:g1}
\Delta \left[ \int_0^{T} \left[ \phi(v(\cdot,t)) - \phi(u(\cdot,t)) \right] \hat{e}(t) \, dt \right] \ge \int_0^{T} \left[ v(\cdot,t) - u(\cdot,t) \right] \hat{e}(t) \, dt \, ,
\end{equation}
still in the very weak sense over $ M $, where we set
$$
\hat{e}(t) :=
\begin{cases}
\frac{e^{-t}}{1-e^{-T}} & \text{if } T < + \infty \, , \\
e^{-t} & \text{if } T=+\infty \, .
\end{cases}
$$
Note that, rigorously, these computations can be justified by taking a separable test function of the form $ \xi(x,t) = \hat{e}(t) \, \eta(x)  $ in \eqref{def-sol} and cutting it off at $ t=T $, where $ \eta \in C^2_c(M) $, with $ \eta \ge 0 $, is arbitrary. Clearly, we have that
$$
 -\left(\left\| u \right\|_{L^\infty(M \times (0,T))} \vee \left\| v \right\|_{L^\infty(M \times (0,T))} \right) \le  u \le v \le \left\| u \right\|_{L^\infty(M \times (0,T))} \vee \left\| v \right\|_{L^\infty(M \times (0,T))}  =: \ell < +\infty \, .
$$
As a result, Lemma \ref{concave-mc} is applicable to $ \phi $ in the interval $ I= [-\ell,\ell] $, ensuring the existence of a concave, continuous and strictly increasing function $ F: [0,+\infty) \to [0,+\infty) $ such that (recall that $v \ge u$ and $ \phi $ is nondecreasing)
$$
\phi(v(x,t)) - \phi(u(x,t)) \le F(v(x,t)-u(x,t)) \qquad \text{for a.e. } (x,t) \in M \times (0,T) \, ,
$$
that is, $ F $ being bijective between $ [0,+\infty) $ and itself,
$$
F^{-1}\!\left( \phi(v(x,t)) - \phi(u(x,t)) \right) \le v(x,t)-u(x,t) \qquad \text{for a.e. } (x,t) \in M \times (0,T) \, ,
$$
which entails
\begin{equation}\label{eq:ah1}
\int_0^{T} F^{-1} \!\left( \phi(v(x,t)) - \phi(u(x,t)) \right) \hat{e}(t) \, dt \le \int_0^{T} \left[ v(x,t) - u(x,t) \right] \hat{e}(t) \, dt \qquad \text{for a.e. } x \in M
\end{equation}
upon integration against $ \hat{e}(t) $ over $ (0,T) $. On the other hand, since $ F^{-1} $ is convex and $ \hat{e}(t)  $ is a probability density on $ (0,T) $, by Jensen's inequality we deduce from \eqref{eq:ah1} that
\begin{equation}\label{eq:ah2}
F^{-1} \! \left( \int_0^{T} \left[ \phi(v(x,t)) -  \phi(u(x,t)) \right] \hat{e}(t) \, dt \right) \le \int_0^{T} \left[ v(x,t) - u(x,t) \right] \hat{e}(t) \, dt  \qquad \text{for a.e. } x \in M \, .
\end{equation}
Hence, in view of \eqref{eq:g1} and \eqref{eq:ah2}, it holds
\begin{equation*}\label{eq:g2}
\begin{aligned}
 \Delta \left[ \int_0^{T} \left[ \phi(v(\cdot,t)) - \phi(u(\cdot,t)) \right] \hat{e}(t)  \, dt \right] \ge F^{-1} \! \left( \int_0^{T} \left[ \phi(v(\cdot,t)) -  \phi(u(\cdot,t)) \right] \hat{e}(t) \, dt \right) .
\end{aligned}
\end{equation*}
We have therefore shown that the function
$$
\underline{W}(x) :=  \int_0^{T} \left[ \phi(v(x,t)) - \phi(u(x,t)) \right] \hat{e}(t)  \, dt  \, ,
$$
which is nonnegative (recall that $\phi$ is nondecreasing) and bounded by construction, satisfies
$$
\Delta \underline{W} \ge F^{-1}(\underline{W}) \qquad \text{in } M
$$
very weakly, where $ F^{-1} $ is a strictly increasing and continuous function on $[0,+\infty)$ with $ F^{-1}(0) = 0 $. Moreover, and most importantly, we can also assert that $ \underline{W}  $ is nontrivial: indeed, if $ \underline{W} = 0 $ identically, then from  \eqref{eq:g1} we would infer that
$$
\int_0^{T} \left[ v(\cdot,t) - u(\cdot,t) \right] \hat{e}(t) \, dt \le 0 \, ,
$$
which is a contradiction to the fact that $v \ge u$ and $v$ is essentially different from $u$. Hence, thanks to Theorem \ref{teo-ell}\eqref{jjj}$ \Rightarrow $\eqref{j}, it follows that $ M $ is stochastically incomplete.
\end{proof}

\begin{proof}[Proof of Corollary \ref{cor-pme}]
First of all, one sees that the proof of Theorem \ref{teo-par}\eqref{i}$ \Rightarrow $\eqref{ii} can be reproduced verbatim with time origin shifted to $s$ and $ S=t-s $, yielding the two-sided estimate
\begin{equation}\label{eee2-shift}
 (t-s) \left( - c \, \varepsilon + \beta \right)  \le	\int_s^t \phi(u(\cdot,\tau)) \, d\tau  \le (t-s) \left( \kappa \, \varepsilon + \beta \right) \qquad \text{in } \Omega_\varepsilon \, ,
\end{equation}
where $ \Omega_\varepsilon $ has been defined in \eqref{def-omega-eps}. Let $ \beta $ be any value in $ \phi(\mathbb{R}) \setminus \left\{ \phi(0) \right\} $, and let $ \varepsilon>0 $ be so small that
\begin{equation}\label{small-eps-1}
(c \vee \kappa) \, \varepsilon <   \frac{\left| \beta - \phi(0) \right|}{2} \, .
\end{equation}
Given an arbitrary smooth domain $ E \Subset M $, we observe that the set
$$
E^c \cap \Omega_\varepsilon
$$
necessarily has positive measure. Indeed, if this was not the case, it would imply that $ \Omega_\varepsilon \subseteq E $ up to a negligible set and therefore
$$
V \le 1-\varepsilon \qquad \text{on } \partial E \, ,
$$
which in turn entails $  V \le 1-\varepsilon$ in the whole $ E $ as $V$ is subharmonic, a clear contradiction to the definition of $ \Omega_\varepsilon $. As a result, thanks to \eqref{eee2-shift} and \eqref{small-eps-1}, we can assert that
$$
\int_s^t \phi(u(\cdot,\tau)) \, d\tau \neq (t-s) \, \phi(0) \qquad \text{a.e. in } E^c \cap \Omega_\varepsilon \, ,
$$
whence property \eqref{lack-supp} follows thanks to the arbitrariness of $E$.
Note that the latter is enjoyed by the infinitely many solutions constructed in Theorem \ref{teo-par}, except at most the one corresponding to $ \beta=\phi(0) $.
\end{proof}

\begin{oss}[When $M$ is not connected]\rm \label{non-conn}
	In the above proofs, the connectedness of $M$ is used only when we appeal to the existence of a regular exhaustion $ \{ E_k \} \subset \mathcal{P}(M)$. On the other hand, even if $M$ is not connected, it is plain that it is stochastically incomplete if and only if it admits at least one stochastically incomplete (hence noncompact) connected component. Therefore, there is no loss of generality in assuming that $M$ is in addition a connected manifold.
\end{oss}

\section*{Declarations}

\noindent \textbf{Competing interests.} The four authors declare that none of them has any competing interests.

\medskip

\noindent \textbf{Funding.}  G.G., M.M.~and F.P.~were financially supported by the projects ``Partial Differential Equations and Related Geometric-Functional Inequalities'' (grant no.~20229M52AS) and ``Geometric-Analytic Methods for PDEs and Applications (GAMPA)'' (grant no.~2022SLTHCE), funded by European Union -- Next Generation EU within the PRIN 2022 program (D.D.~104 - 02/02/2022 Ministero dell’Università e della Ricerca -- Italy). K.I.~was financially supported by JSPS KAKENHI Grant Number 19H05599 (Japan).

\end{document}